\documentclass[aip, jmp, preprint, secnumarabic, superscriptaddress, citeautoscript, showkeys, longbibliography, nofootinbib]{revtex4-2}

\usepackage{cmap} % make ligatures etc. copy- and searchable
\usepackage[T1]{fontenc} % needed to have copy-able umlaute
\usepackage{lmodern}
\usepackage{amsmath,amssymb,mleftright}
\usepackage[version=3]{mhchem} % chemical typesetting
\usepackage{graphicx}% Include figure files
\graphicspath{{.}{./figures/}}

\usepackage{enumitem}
\setlist[enumerate]{label={(\roman*)}, nolistsep}

\usepackage{amsthm}
\theoremstyle{definition}

\newtheorem{remark}{Remark}

\newtheorem{proposition}{Proposition}

\usepackage{booktabs}

\usepackage{bm} % bold math

\usepackage{url,hyperref}
\usepackage[table,usenames,dvipsnames]{xcolor}
\hypersetup{colorlinks=true, linkcolor=BrickRed,
  urlcolor=blue!50!black, citecolor=blue!50!black}

\usepackage[capitalise]{cleveref}
\crefrangelabelformat{equation}{(#3#1#4)$-$(#5#2#6)}
\crefformat{proposition}{Prop.~#2#1#3}

\usepackage{siunitx}
\DeclareSIUnit\Molar{M}

\usepackage{tikz-cd} % typesetting diagrams
\usepackage{chemarr}

\renewcommand\epsilon{\varepsilon}
\renewcommand\phi{\varphi}
\renewcommand\theta{\vartheta}
\renewcommand\rho{\varrho}
\newcommand\diff{\mathrm{d}}
\renewcommand\geq\geqslant
\renewcommand\leq\leqslant

\usepackage[normalem]{ulem}
\begin{document}

%\AtEveryBibitem{\printfield{note}\clearfield{note}\item}

\title{Unfolding the geometric structure and multiple timescales of the urea--urease pH oscillator}

\newcommand\zib{\affiliation{Zuse Institute Berlin, Takustra{\ss}e 7, 14195 Berlin, Germany}}
\newcommand\fub{\affiliation{Freie Universit{\"a}t Berlin, Department of Mathematics and Computer Science, Arnimallee 6, 14195 Berlin, Germany}}
\newcommand\amst{\affiliation{KdVI Institute for Mathematics, University of Amsterdam, The Netherlands}}
\newcommand\pots{\affiliation{Universität Potsdam, Institut für Mathematik, Karl-Liebknecht-Straße 24, 14476 Potsdam, Germany.}}

\author{Arthur V. Straube} \zib \fub \email{straube@zib.de}
\author{Guillermo Olicón Méndez} \pots
\author{Stefanie Winkelmann} \zib
\author{Felix Höfling} \fub \zib
\author{Maximilian Engel} \fub \amst

\begin{abstract}
We study a two-variable dynamical system modeling pH oscillations in the urea–urease reaction within giant lipid vesicles --- a problem that intrinsically contains multiple, well-separated timescales. Building on an existing, deterministic formulation via ordinary differential equations, we resolve different orders of magnitude within a small parameter and analyze the system's limit cycle behavior using geometric singular perturbation theory (GSPT). By introducing two different coordinate scalings --- each valid in a distinct region of the phase space --- we resolve the local dynamics near critical fold points, using the extension of GSPT through such singular points due to Krupa and Szmolyan. This framework enables a geometric decomposition of the periodic orbits into slow and fast segments and yields closed-form estimates for the period of oscillation. In particular, we link the existence of such oscillations to an underlying biochemical asymmetry, namely, the differential transport across the vesicle membrane.
\end{abstract}


\maketitle

\section{Introduction}

Oscillatory rhythms are ubiquitous in biophysical systems, underpinning essential functions such as spatiotemporal self-organization, cellular signaling, metabolism, and homeostasis. \cite{epstein:PD1991,goldbeter:book1996,epstein:book1998,novak:NRMCB2008,deKepper:book2009,yashin:RPP2012,goeth:CP2025} 
These dynamics are often mediated by enzymes, which play a central role in intracellular kinetics and typically operate on multiple timescales, embedded in complex feedback mechanisms that regulate their activity.\cite{goldbeter:FEBSL1993,barkai:N1997,cornish-bowden:book2004,alon:book2007,rubin:JCP2016,milster:PRE2024} Enzymatic activity and cellular function are highly sensitive to environmental factors, in particular to the level of acidity or pH.\cite{alberty:BBA1954,cornish-bowden:book2004,fidaleo:CBEQ2003,casey:NRMCB2010,yashin:RPP2012} Systems exhibiting periodic changes in pH are known as \textit{pH oscillators}.\cite{rabai:ACHMC1998,mcilwaine:CPL2006,orban:ACR2015,horvath:APR2018,duzs:CSC2023} While many studies of such oscillators have focused on closed reactors, recent work has increasingly turned to open reaction compartments that exchange matter with their surroundings, \cite{miele:Proc2016,miele:LNBE2018,muzika:PCCP2019,straube:JPCL2021,miele:JPCL2022,leathard:Ch23,ridgway-brown:ACSO25,itatani:CRRS2025} enabling the sustained operation of chemical switches or clocks.\cite{lente:NJC2007,hu:JPCB2010,bubanja:RKMC2018} In this work, we analyze the urea--urease reaction confined to a lipid vesicle---a minimal model system in which urease catalyzes the hydrolysis of urea in a pH-dependent manner, while the vesicle membrane permits exchange of acid and substrate with the surrounding medium. This biologically and technologically relevant setting\cite{mai:CC2021,galanics:CPC2024,bashir:MSDE2024,ivanov:ACIE2025,mahmud:ACSO2025} gives rise to fast–slow oscillatory dynamics that can be captured by low-dimensional models.\cite{bansagi:JPCB2014,straube:JPCB2023} 
Controlling the timescale separation, we uncover the geometric structure underlying the resulting limit cycle dynamics using methods from geometric singular perturbation theory (GSPT) \cite{kuehn:book2015, wechselberger:book2020} which has proven to be a powerful mathematical tool in the context of (bio)chemical reactions \cite{engel:A2023, Gucwa:DCDS2009, Kosiuk:SIADS2011, kuehn:JNS2015}.

\paragraph*{The model.}

We study the two-variable ODE model derived in Ref.~\citenum{straube:JPCB2023}, which deterministically describes pH oscillations in the urea–urease reaction confined to a giant lipid vesicle. Further reduction of this model (\cref{sec:realistic-model}) leads to the following dynamical system in the dimensionless variables $s \geq 0$ and $h \geq 0$, representing rescaled concentrations of the substrate molecules \ce{S} and hydrogen ions \ce{H+}, respectively:
\begin{subequations} \label{eq:rre}
    \begin{align}
    \frac{ds}{dt} & = f(s,h) := -r(h)  s + K_s\,,  \label{eq:rre-s} \\
    \frac{d h}{dt} & = g(s, h) := - q(s,h) + K_h(1-h)\,,\label{eq:rre-h}
    \end{align}
\end{subequations}
where $K_s$, $K_h>0$ are positive constants. The rate function $r(h)$, which determines the consumption of the substrate $s$, has the form
\begin{equation}
  r(h) = \bigl(\beta \epsilon_1/h + 1 + \beta h/\epsilon_1 \bigr)^{-1} >0,
  \label{eq:fun-r(h)}
\end{equation}
with a small parameter $\epsilon_1 > 0$ and some constant $\beta > 0$. 
This function exhibits a single maximum and tends to zero as $h\to \infty$ or $h\to 0$.
The consumption of hydrogen $h \geq 0$ is governed by the reaction speed $q(s,h)$, which is given as the non-negative root of the quadratic equation $q^2 + v(h) q - K\epsilon_2^{-1} r(h) h^2 s = 0$:
\begin{align}
    q(s,h)  =  - \frac{1}{2} v(h) + \frac{1}{2}\sqrt{v(h)^2+4 K \epsilon_2^{-1} r(h) h^2 s} \geq 0 \,, \label{eq:fun-q(s-h)}
\end{align}
where $v(h)$ is defined as
\begin{align}
    v(h) & := \alpha K \epsilon_2^{-1} h^2 -K_h(1-h) \, ,\label{eq:v(h)} 
\end{align}
with constants $\epsilon_2$, $\alpha$, $K$, $K_h >0$.
We note that for fixed $s$, the function $q(s,h)$ vanishes as $h\to \infty$, while as $h \to 0$ it asymptotically behaves like $-v(h)$.
Second, $\epsilon_1$ appears only in the combination $h/\epsilon_1$ and $\epsilon_2$ only in $h^2/\epsilon_2$; this observation has been motivating for the approach followed below.
An exemplary numerical solution of the ODE system~\eqref{eq:rre} is plotted in \cref{fig:limit-cycle-start}(a).

\begin{figure}
    \includegraphics[width=0.9\textwidth]{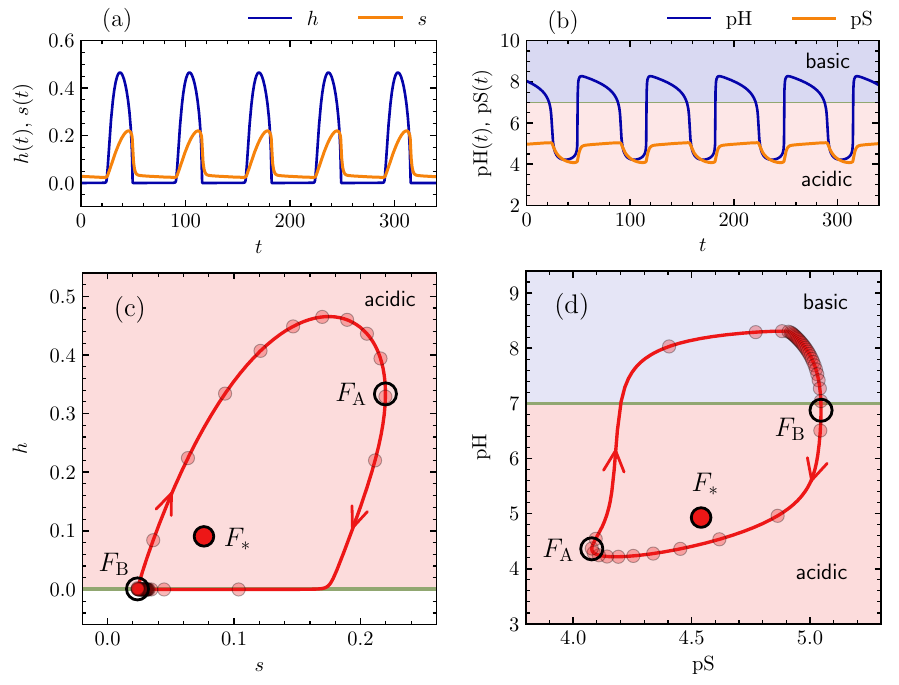}
    \caption{
    Panels (a) and (b) show a typical numerical solution of \cref{eq:rre} for the parameter values listed in \cref{table:reduced-parameters}; panels (c) and (d) depict the corresponding limit cycles for the same data (red solid lines).
    Left panels [(a) and (c)] use the original $(s,h)$ variables, while right panels [(b) and (d)] employ the negative logarithms $(\ce{pS},\ce{pH})$, commonly used in chemistry \cite{straube:JPCL2021,straube:JPCB2023}, see  \cref{{eq:logSH}}.
    In the bottom panels, arrows indicate the direction of motion, and disks show the evolution of the variables sampled equidistantly in time.
    The limit cycle encloses the only fixed point (dark red disc in panels (c) and (d)), which is repelling and located at $F_*=(s_*,h_*)\approx (0.0762, 0.0906)$ and $(\ce{pS}_*,\ce{pH}_*)\approx (4.54, 4.93)$, respectively.
    The turning points $F_\textup{A}$ and $F_\textup{B}$ mark those points on the limit cycle with the largest and smallest values of $s$ and the smallest and largest values of $\ce{pS}$, respectively. Note that smaller values of \ce{pS} and \ce{pH} correspond to larger values of $s$ and $h$ and vice versa, which effectively flips the direction of the axes and changes the relative location of points $F_\textup{A}$ and $F_\textup{B}$.
    The horizontal green line represents the neutral value, $\ce{pH} = 7$
    ($h=h_\textup{max}$, see \cref{ssec:ode-lim-cycle}), separating acidic ($\ce{pH} < 7$, $h > h_\textup{max}$, shaded light red) and basic ($\ce{pH} > 7$, $h < h_\textup{max}$, shaded light blue in panels (b) and (d)) domains.
    }
    \label{fig:limit-cycle-start}
\end{figure}

\paragraph*{Biochemical background and interpretation.} 

System \eqref{eq:rre} models the urease-assisted catalysis of urea, serving as the substrate \ce{S}, with the reaction rate depending on the hydrogen ion (or proton) concentration, i.e., on the pH value quantifying the acidity. Urease, an enzyme acting as a biological catalyst, regulates the speed of the catalytic reaction via the term $r(h)s$, where $r(h)$ is the hydrogen-dependent reaction rate. The effective decay of hydrogen \cite{bansagi:JPCB2014,straube:JPCB2023} is governed by the function $q(s,h)$, which captures the nonlinear coupling between substrate and proton concentrations.

The described process occurs in an open reactor -- a unilamellar vesicle enclosed by a membrane partially permeable to chemical species. Exchange with the surrounding medium ensures a continuous supply of both substrate and acid, with influx rate constants $K_s$ and $K_h$, respectively. Inside the vesicle, the core reaction acts as a chemical switch or clock,\cite{lente:NJC2007,hu:JPCB2010,bubanja:RKMC2018} rapidly driving the system from acidic to basic conditions. In turn, the external inflow of acid and substrate resets the pH clock after each cycle, enabling sustained oscillations. These oscillations result from the interplay between internal switching dynamics and external replenishment, and are characterized by fast--slow transitions between acidic and basic phases. A detailed derivation of the simplified, dimensionless system \eqref{eq:rre} from a previously reduced model of the urea--urease reaction network\cite{straube:JPCB2023} is presented in \cref{sec:realistic-model}.

Previous studies have shown that the system under consideration---like those in Refs.~\citenum{bansagi:JPCB2014,straube:JPCB2023}---captures a practically relevant and experimentally realistic setting; for further insights, see Ref.~\citenum{acar:CSBB2025}. One of the strengths of our study is to build on this foundation and rigorously analyze the model's complex oscillatory behavior. However, due to this realism and unlike many mathematically idealized systems, the present  model does not feature a single small parameter that would directly justify a standard fast–slow decomposition or suggest an obvious timescale separation. For parameter regimes relevant to experiments, numerical integration of system~\eqref{eq:rre} reveals stable limit cycle solutions, characterized by very different time scales along the segments of the cycle (\cref{fig:limit-cycle-start}). This timescale separation manifests itself in the form of mixed-mode oscillations, which fall outside the direct applicability of classical singular perturbation theory.
In particular, the dynamics is governed by multiple small dimensionless parameter combinations, none of which alone controls the system’s behavior. 

To address this obstacle to analysis via GSPT\cite{kuehn:book2015,wechselberger:book2020}, 
we identify a suitable formalization of the timescale separation that emerges from the interplay of multiple small parameters. This enables the application of GSPT, combining the conventional Fenichel's theory for normally hyperbolic critical manifolds\cite{fenichel:JDE1979} with its extension to nonhyperbolic points, such as folds and canards, developed via blow-up techniques.\cite{krupa:SIAM-JMA2001} Using this framework, we characterize the structure of the limit cycle and quantify its decomposition into alternating fast and slow segments. Two key components of the limit cycle are stable branches of distinct critical manifolds---one under acidic pH conditions, the other under basic pH conditions---each requiring a different rescaling to resolve the local dynamics. We complete the global picture of the oscillation mechanism by constructing transition maps between suitable Poincaré sections.

In \cref{ssec:ode-system}, we present initial observations, establish the existence of a limit cycle, and introduce a rescaling in terms of a single small parameter. The existence of a limit cycle was conjectured in earlier work,\cite{straube:JPCB2023} based on the presence of a single repelling fixed point and an application of the Poincaré–Hopf index theory. Here, we provide a more rigorous proof using the Poincaré–Bendixson theorem. \cref{sec:analysis} contains the core of our analysis: we apply GSPT to resolve the dynamics near the fold points at low and neutral pH (\cref{sec:acid} and \cref{sec:neutral}), construct the global structure of the limit cycle and compare with numerical solutions (\cref{sec:global-picture}). We discuss the implications for the chemical system at hand (\cref{sec:implications})
and conclude in \cref{sec:conclusion}.

\section{Initial Observations and Analytical Setup} \label{ssec:ode-system}

In this section, we briefly discuss a parameter regime where \cref{eq:rre} exhibits oscillatory solutions and we give a necessary conditions for the existence of a limit cycle.
We suggest a coupling of the small parameters $\epsilon_1$ and $\epsilon_2$ 
to proceed with a single small parameter $\epsilon$, which allows converting  system \eqref{eq:rre} into a fast-slow system in standard form.

\subsection{Observation of sustained oscillations} \label{ssec:ode-lim-cycle}

\Cref{fig:limit-cycle-start} shows a typical limit cycle solution of system \eqref{eq:rre}, using the variables $(s,h)$ and, alternatively, their negative logarithms, $(\ce{pS},\ce{pH})$, employed in biochemistry (see  \cref{eq:logSH} for the conversion relation).
From the representations in \cref{fig:limit-cycle-start}, one anticipates the presence of multiple timescales:
in panels (a) and (b), kinks and rapid variations in time
and, in panels (c) and (d), an uneven spacing of markers placed equidistantly in time along the solution, especially in the vicinity of point $F_\textup{B}$.
The chosen parameter values are listed in \cref{table:reduced-parameters} in \cref{sec:realistic-model} and are the same as in a previous study \cite{straube:JPCB2023}, approximately resembling the experimental situation:
while $\beta$ is fixed by the properties of the urease enzyme and
$K_s$ and $K_h$ are given by the specific permeabilities of the vesicle membrane,
$\alpha$, $\epsilon_1$, and $\epsilon_2$ can be controlled through the external concentrations of substrate and hydrogen ion, outside of the vesicle.

The position $h_\textup{max}$ of the maximum of the consumption rate $r(h)$ (see \cref{eq:fun-r(h)}) is determined by $r'(h_\textup{max})=0$, yielding $h_{\text{max}}=\epsilon_1$. For the given parameters, it corresponds to the neutral value of $\ce{pH}=7$ (green horizontal lines in Figs.~\ref{fig:limit-cycle-start}(b), \ref{fig:limit-cycle-start}(d)).
At $h=h_{\text{max}}$, the conversion of substrate into product is most efficient and any deviation from the optimal level of $\ce{pH}$ slows down the corresponding reaction rate $r(h)$. The limit cycle crosses this line and one part of it belongs to the acidic domain (shaded light red) with $\ce{pH} < 7$ ($h>h_{\text{max}}$) and part to the basic domain (shaded light blue), where $\ce{pH} > 7$ ($h<h_{\text{max}}$). 
Using linear scales for the variables $s$ and $h$, as in \cref{fig:limit-cycle-start}(a) and (c), the basic region is not well resolved: the neutral value, $\ce{pH} = 7$, corresponds to $h=h_{\text{max}} \approx \num{7.7e-4}$), being compressed and pushed against the $h=0$ axis. Accordingly, the turning point $F_\textup{A}$ lies deeply in the acidic domain, while the turning point $F_\textup{B}$ is located in the vicinity of the neutral \ce{pH}. As we will show below, the position of the turning point $F_\textup{B}$, which underlies the geometric structure of the limit cycle, is determined exactly by the maximum of the function $r(h)$.

\subsection{Existence of the limit cycle} \label{ssec:lim-cycle-exist}

Generally, the system under consideration, \cref{eq:rre}, exhibits either a steady-state solution or oscillations at long times. Focusing on the limit cycle, we briefly outline the region of parameters where sustained oscillations take place.

The $s$- and $h$-nullclines of system \eqref{eq:rre}, denoted as $\mathcal{N}_s$ and $\mathcal{N}_h$, respectively, are parametrized as
\begin{subequations} \label{eq:rre-nc}
\begin{align}
    \mathcal{N}_s = \{ (n_s(h), h): h \geq 0 \}, &&  n_s(h) &:= \frac{K_s}{r(h)}\,,  \label{eq:rre-nc-s} \\
\intertext{and}
    \mathcal{N}_h = \{ (n_h(h), h): h \geq 0 \}, &&  n_h(h) &:=  \frac{\alpha K_h(1-h)}{r(h)}  \,. \label{eq:rre-nc-h} 
\end{align}
\end{subequations}
The unique equilibrium $F_*=(s_*,h_*) \in \mathbb{R}^2_{>0}$ solves $n_h(h_*)=n_s(h_*)=:s_*$ for
\begin{equation}
  s_* = \frac{K_s}{r(h_*)} > 0, \qquad h_* = 1 - \frac{K_s}{\alpha K_h} > 0\,. \label{eq:fixp}
\end{equation}
For the parameter values in \cref{table:reduced-parameters}, used to generate data in \cref{fig:limit-cycle-start}, we obtain
$F_*=(s_*,h_*)\approx (0.0762, 0.0906)$ or $(\ce{pS}_*,\ce{pH}_*)\approx (4.54, 4.93)$.

\begin{remark}\label{rem:aK}
    In the following, we require that $\alpha K_h > K_s$, so that $F_*$ exists in the positive quadrant, in particular, the concentration $h_*$ is positive.
    We note that this condition on the parameters aligns with a general claim made in a previous study for the existence of a limit cycle---namely, differential transport of substrate and hydrogen across the membrane \cite{bansagi:JPCB2014}, which will be discussed further in \cref{sec:conclusion}.
    The ratio $K_h/K_s$ quantifies the difference in the transport of the hydrogen ions and substrate across the membrane, and $(2\alpha)^{-1}=[\ce{S}_\textrm{ext}]/[\ce{H+}_\textrm{ext}]$ represents the ratio of their external concentrations, see \cref{sec:realistic-model}.
\end{remark}

\begin{figure}
    \includegraphics[width=0.47\textwidth]{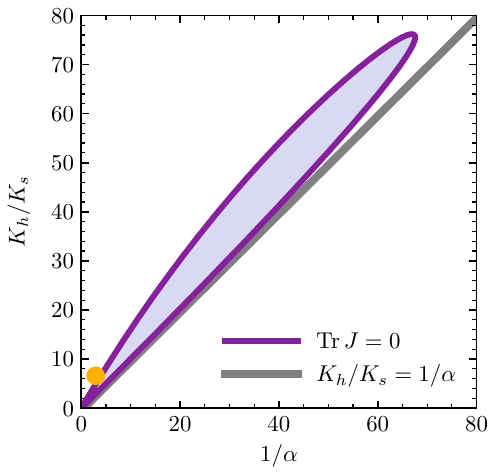} 
    \caption{
    Numerical results for the domain of oscillation in the parameter combinations $(K_h/K_s, \alpha^{-1})$ at fixed values of $\beta$, $\epsilon_1$, and $\epsilon_2$ (\cref{table:reduced-parameters}). 
    The curve $\mathrm{Tr}\, J = 0$ (violet line) separates the regions of stable (attractive) and unstable (repelling) behavior of the fixed point. 
    Oscillations occur in the blue-shaded region, where the fixed point is a repelling node or focus, $\mathrm{Tr}\, J > 0$. 
    The yellow disc marks the parameter values of the real-world system\cite{bansagi:JPCB2014,straube:JPCB2023} (see \cref{table:parameters,table:reduced-parameters}), for which the limit cycle is shown in \cref{fig:limit-cycle-start}. A lower bound for oscillations is given by the straight line, $K_h/K_s = \alpha^{-1}$, see \cref{eq:fixp}.
    }
    \label{fig:stability}
\end{figure}

The stability of the equilibrium point is determined by
the Jacobian matrix of \cref{eq:rre} at $F_*$, which we denote as 
$J=\partial (f,g)/\partial (s,h)\vert_{F_*}$. The trace $\mathrm{Tr}\, J$ exhibits a sign change, suggesting that the system admits a \textit{Hopf bifurcation} \cite{kuznetsov:book2004}.
\Cref{fig:stability} depicts the numerically determined domain of oscillations in the parameter plane $(K_h/K_s, 1/\alpha)$; outside this domain, the orbits of the system converge to the unique steady state, provided $\alpha K_h>K_s$.
The domain of oscillations in the parameter space is bounded from below by the condition $K_h/K_s = \alpha^{-1}$. 

\begin{figure}
    \includegraphics[width=0.55\textwidth]{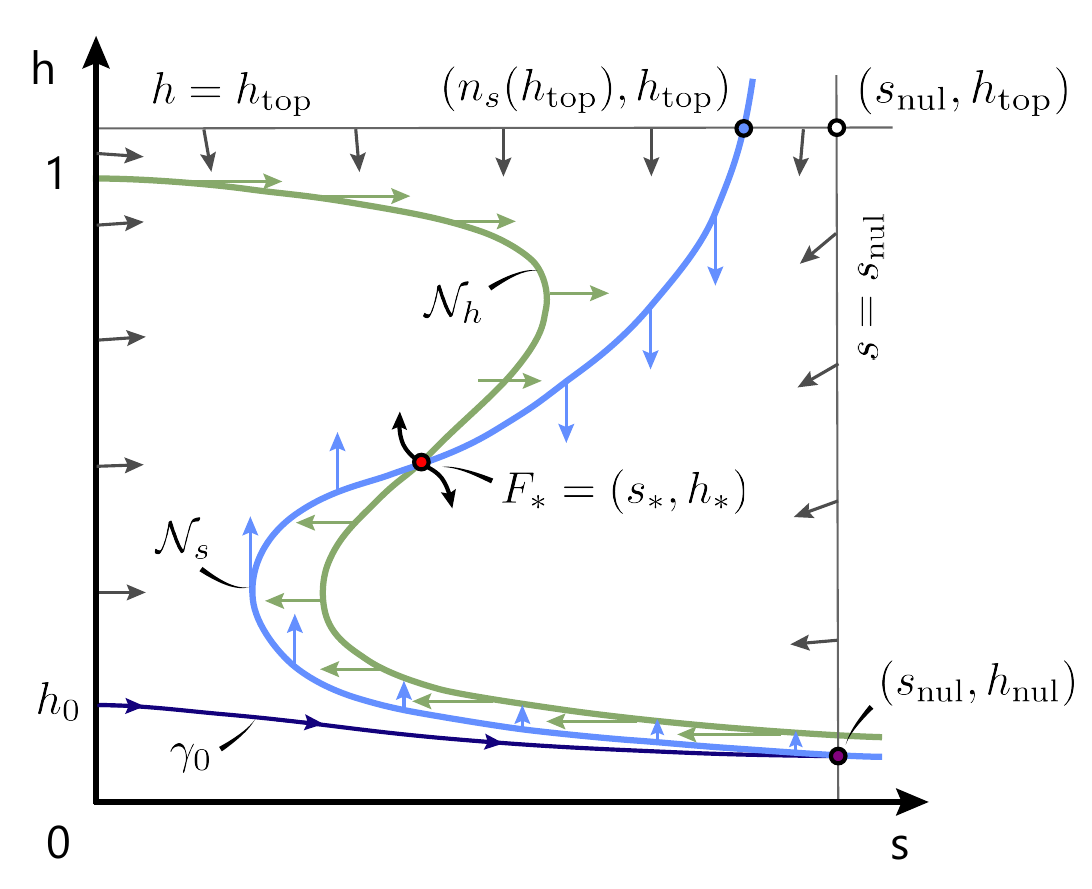} 
    \caption{
    A schematic supporting the proof of \cref{PROP:existence_cycle}. The $s$- and $h$-nullclines, $\mathcal{N}_s$ and $\mathcal{N}_h$, are shown as light blue and green lines, respectively; arrows indicate the direction of the flow field at the corresponding nullclines. The top, $h=h_\textup{top}$, and right, $s=s_\textup{nul}$, border of the compact region are given by the horizontal and vertical straight lines, respectively. The bottom border is given by the semiorbit $\gamma_0$ starting at $(0,h_0)$ and ending at $(s_\textup{nul}, h_\textup{nul})$ and is plotted in dark blue. Other characteristic intersection points, including the fixed point $F_*=(s_*,h_*)$, are shown as colored circle markers.  
    }
    \label{fig:exist}
\end{figure}

By means of standard phase-plane analysis, the existence of an attracting limit cycle of \eqref{eq:rre} is guaranteed as long as the equilibrium $F_*$ is unstable, as stated in the following proposition:

\begin{proposition}
    \label{PROP:existence_cycle}
Consider system~\eqref{eq:rre} with $\alpha K_h>K_s$ such that  $F_*$ given by \cref{eq:fixp} is the equilibrium. If the parameter values are such that $\textup{Tr}\, J\vert_{F_*}$ and  $\textup{Det}\, J\vert_{F_*}$ are positive, then system~\eqref{eq:rre} admits an attracting limit cycle contained in the positive quadrant $\mathbb{R}^2_{\geq 0} := \{(s,h):s,\,h\geq 0\}$. 
\end{proposition}

\begin{proof} 
    We construct a positively invariant compact region $\mathcal{R}$ contained in $\mathbb{R}^2_{>0}$. First, let us fix any $h_\textup{top}>1$, an arbitrarily large $s_\textup{nul}>\max\left\{K_s/r(h_\textup{top}), s_*\right\}$, and the point $(s_\textup{nul},h_\textup{nul})\in \mathcal{N}_s$. Choose the unique initial condition $(0,h_0)$ such that its forward semiorbit $\gamma_0:=\{(s(t),h(t)) : t\geq 0, \;(s(0),h(0))=(0,h_0)\}$ intersects $\mathcal{N}_s$ at $(s_\textup{nul},h_\textup{nul})$.  
    Consider the compact region $\mathcal{R}\subset \mathbb{R}^2_{>0}$ delimited by $\gamma_0$, $\{s=s_\textup{nul}\}$, $\{h=h_\textup{top}\}$, and $\{ s=0\}$, which is positively invariant for sufficiently large $s_\textup{nul}$, see \cref{fig:exist}. Indeed, the vector field given by \cref{eq:rre} points inwards on $\partial \mathcal{R}$, apart from $\gamma_0$ on which it is tangent. This follows from the fact that
    \begin{align*}
        \dot{s} & > 0, \textup{ when } s=0, \\   
        \dot{s} & <0,  \textup{ when } s=s_\textup{nul}
 \textup{ as } s_\textup{nul}> \frac{K_s}{r(h)}, \textup{ for all } h\in[h_\textup{nul},h_\textup{top}], \\ 
\dot{h} & \leq K_h(1-h) <0 \textup{ when $h=h_\textup{top}>1$}.
    \end{align*}
    Since $F_*$ is an unstable node or focus, as indicated by $\textup{Tr}\, J\vert_{F_*}>0$ and  $\textup{Det}\, J\vert_{F_*}>0$, the result follows from the Poincaré--Bendixson theorem (cf.~e.g.~Ref.~\citenum{CoddingtonLevinson55}~[Chapter 16, Theorem 2.1]).
\end{proof}

\subsection{Introduction of a single small parameter} \label{ssec:ode-timesc-sep}

We have seen that, for a suitable choice of parameter values, system \eqref{eq:rre} admits an attracting limit cycle, which displays slow and fast regimes. We will demonstrate below that the time scale separation emerges 
in a suitable limit process for the small parameters $\epsilon_1$ and $\epsilon_2$. 
One observes that $\epsilon_1^2$ and $\epsilon_2$ are of similar order of magnitude (see \cref{table:reduced-parameters}, and by the only occurrence of $\epsilon_1$ and $\epsilon_2$ through the combinations $h/\epsilon_1$ and $h^2/\epsilon_2$ in the model, \cref{eq:fun-r(h),eq:fun-q(s-h),eq:v(h)}. 
This structure suggests to use a scaling $\epsilon_1=\mathcal{O}(\epsilon)$ and $\epsilon_2=\mathcal{O}(\epsilon^2)$ as $\epsilon\rightarrow 0$ such that
$h/\epsilon_1$ and $h^2/\epsilon_2$ can remain balanced if $h = \mathcal{O}(\epsilon)$.
Concretely, we set
\begin{equation}
  \epsilon_1 =: C \epsilon\,  \quad \text{and} \quad \epsilon_2 =: A^{-1} \epsilon^{2}\,, \label{eq:resc-constants}
\end{equation}
with the constant $C \approx 0.77$ chosen for convenience such that $\epsilon = \num{e-3}$ for the parameters of the real-world system in \cref{table:reduced-parameters}; it follows that $A \approx 1.4$. Upon substituting in \cref{eq:rre}, the system reads
\begin{subequations} \label{eq:rre-eps}
    \begin{align}
    \frac{ds}{dt} & =  -r_\epsilon(h)  s + K_s\,,  \label{eq:rre-eps-s} \\
    \frac{d h}{dt} & =  - q_\epsilon(s,h) + K_h(1-h)\,, \label{eq:rre-eps-h}
    \end{align}
\end{subequations}
with the functions
\begin{subequations}
    \label{eq:r-q-eps}
    \begin{align} 
        r_\epsilon(h) & := \frac{1}{\beta C (h/\epsilon)^{-1} + 1 + \beta C^{-1} (h/\epsilon)},  \label{eq:r-eps} \\
        q_\epsilon(s, h)  & :=  \frac{1}{2}\sqrt{v_\epsilon(h)^2+ 4 A K r_{\epsilon}(h)(h/\epsilon)^2 s}- \frac{1}{2} v_\epsilon(h), \label{eq:q-eps} 
    \end{align}
\end{subequations}
where $v_\epsilon(h)$ is the quadratic function
\begin{align}
    v_\epsilon(h) & := \alpha A K (h/\epsilon)^2 -K_h(1-h) \,, \label{eq:v-eps}
\end{align}
as induced by $v(h)$ defined in \cref{eq:v(h)}.
The equilibrium of \cref{eq:rre-eps} is calculated analogously to \cref{eq:fixp} and is located at
$F_{*,\epsilon} := (s_{*,\epsilon}, h_{*,\epsilon}) = (K_s/r_\epsilon(h_*), h_*)$; in particular, $h_{*,\epsilon}=h_*$ is independent of $\epsilon$, but $s_{*,\epsilon} = O(\epsilon^{-1})$ diverges as $\epsilon \to 0$.

\section{Fast-slow analysis of the system} \label{sec:analysis}

For the fast-slow analysis of the pH oscillator, described by the dynamical system \eqref{eq:rre-eps}, we employ GSPT techniques \cite{kuehn:book2015, wechselberger:book2020} to obtain and characterize the critical manifold, which pertains to the singular limit ($\epsilon = 0$) and reveals the geometric structure underlying the limit cycle for $\epsilon$ small enough.
\Cref{fig:limit-cycle-start} illustrates the limit cycle for the original system, which, after rescaling of constants, corresponds to $\epsilon = \num{e-3}$;
decreasing $\epsilon$ further enhances the timescale separation.
Numerical integration of \cref{eq:rre-eps} shows that, while the limit cycle deforms, it remains bounded in the $h$-variable, but progressively extends in the $s$-variable. However, when represented in terms of the rescaled variable $\sigma = \epsilon s$, see \cref{fig:scheme}(a), the limit cycle becomes bounded and appears to converge to a specific shape as $\epsilon$ is decreased. This observation motivates us, in \cref{sec:acid}, to perform the fast-slow analysis in the variables $(\sigma, h)$. It reveals the slow and fast motions in the acidic domain ($h > h_\textup{max}$), including the fold point $F_\textup{A}$, but it cannot resolve the fold point $F_\textup{B}$ approached from the basic domain ($h < h_\textup{max}$), which is close to the $h=0$ line.
We address this issue in \cref{sec:neutral} by introducing another rescaling, $(s, \eta)$ with $\eta= h/\epsilon$, which can capture the dynamics near $F_\textup{B}$, but not in the acidic regime.

\begin{figure}
    \includegraphics[width=0.9\textwidth]{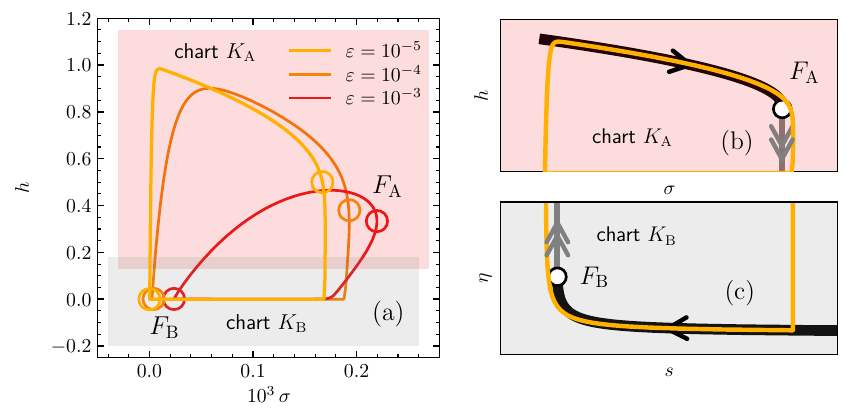}
    \caption{
    (a): Numerical solutions of \cref{eq:rre-eps} for $\epsilon = \num{e-5}, \num{e-4}, \num{e-3}$ showing the limit cycle trajectories and their fold points in the variables ($\sigma, h$) with $\sigma = \epsilon s$. Shaded regions depict the mutual arrangement of the charts $K_\textup{A}$ and $K_\textup{B}$ associated with the two rescalings, as described in the text.
    (b) and (c): Charts $K_\textup{A}$ in the variables $(\sigma, h)$ and $K_\textup{B}$ in the variables $(s, \eta)$, showing the transition from their slow manifolds (solid black lines) to fast fibers (solid gray lines) at the fold points $F_\textup{A}$ and $F_\textup{B}$ (open circles) for $\epsilon=0$, respectively; the variables are connected through $\sigma =\epsilon s$ and $h = \epsilon \eta$. The yellow solid line is an $\epsilon$-perturbed trajectory ($\epsilon=\num{e-5}$), which illustrates how  switching between the charts leads to a closed orbit.
    }
    \label{fig:scheme}
\end{figure}

The two rescalings provide different perspectives on the limit cycle, corresponding to two charts $K_\textup{A}$ and $K_\textup{B}$, which overlap and can be matched, as depicted in \cref{fig:scheme}.
In both charts, we first focus on the singular limit $\epsilon=0$ and identify the critical manifolds, which are \textit{normally hyperbolic} away from the points $F_\textup{A}$ and $F_\textup{B}$. 
These invariant manifolds persist as invariant slow manifolds for $\epsilon$ small enough, and can be extended beyond the generic fold points $F_\textup{A}$ and $F_\textup{B}$ using techniques of GSPT generalized for degenerate equilibria\cite{krupa:SIAM-JMA2001,kuehn:book2015}.
It implies that a trajectory moving rightwards along the slow manifold in chart $K_\textup{A}$ (thick black solid line in \cref{fig:scheme}(b)) passes the fold point $F_\textup{A}$ and transitions to a fast fiber pointing downwards (vertical solid gray line). Similarly, in chart $K_\textup{B}$, a trajectory moving leftwards along the slow manifold passes the fold point $F_\textup{B}$ and quickly moves vertically upwards, see \cref{fig:scheme}(c).
A coordinate transformation matches the dynamics in charts $K_\textup{A}$ and $K_\textup{B}$, as discussed in \cref{sec:global-picture}, ultimately resolving the mixed-mode dynamics of the limit cycle.

\subsection{The fold point \texorpdfstring{$\bm{F}_{\mathbf{A}}$}{FA} approached from the acidic domain, at low \ce{pH}} \label{sec:acid}

For investigating the turning point ${F}_\textup{A}$ located in the acidic domain, at relatively large values of $h$ (low pH level), we use the rescaled variables
\begin{align}
    \sigma = {\epsilon} s\,, \quad \tau:=\epsilon t\,, \label{eq:s-resc-a}
\end{align}
where $\sigma$ is a small variable and $\tau$ is a slow time relative to the original variables $s$ and $t$, respectively. 
With this ansatz, the dynamical system \eqref{eq:rre-eps} becomes
the fast-slow system
\begin{subequations} \label{eq:rre-fs-a}
    \begin{align}
	\frac{d\sigma }{d \tau} & = \tilde{f}_\epsilon(\sigma, h) := - \tilde{r}_\epsilon(h) \sigma + K_s \,,  \label{eq:rre-fs-a-s} \\
	\epsilon \frac{d h}{d \tau} & = \tilde{g}_\epsilon(\sigma, h) := - \tilde{q}_\epsilon(\sigma, h) + K_h (1-h)
        \,,\label{eq:rre-fs-a-h}
    \end{align}
\end{subequations}
with $\tilde{r}_\epsilon(h) := \epsilon^{-1} r_\epsilon(h)$ and 
$\tilde{q}_\epsilon(\sigma, h):={q}_\epsilon(\sigma/\epsilon, h)$.
Inserting the definitions of $r_\epsilon$ and $q_\epsilon$ from \cref{eq:r-eps,eq:q-eps}, we obtain
\begin{equation}
  \tilde{r}_\epsilon(h) = \frac{1}{\epsilon^2 \beta C/h +\epsilon + (\beta/C) h} \label{eq:tilde_r-eps}
\end{equation}
and
\begin{equation}
\tilde{q}_\epsilon(\sigma,h) = 
          -\frac{1}{2 \epsilon^2} \left[\sqrt{\tilde{v}_\epsilon(h)^2 + 4 \epsilon^2 A K \tilde{r}_\epsilon(h) h^2 \sigma} - \tilde{v}_\epsilon(h) \right],
          \label{eq:q-eps-simple}
\end{equation}
where 
\begin{equation}
    \tilde{v}_\epsilon(h) := \epsilon^2 v_\epsilon(h)= \alpha A K h^2 -\epsilon^2 K_h(1-h) \label{eq:tilde:_v-eps}
\end{equation}
for $v_\epsilon(h)$ given in \cref{eq:v-eps}.
Switching back to the timescale $t$, 
we rewrite the fast-slow system \eqref{eq:rre-fs-a} as
\begin{subequations} \label{eq:rre-f-a}
    \begin{align}
	\frac{d \sigma}{d t } & = \epsilon \tilde{f}_\epsilon(\sigma, h)\,,  \label{eq:rre-f-a-s} \\
	\frac{d h}{d t} & = \tilde{g}_\epsilon(\sigma, h)\,.\label{eq:rre-f-a-h}
    \end{align}
\end{subequations}

\subsubsection{Critical manifold and singular dynamics} \label{ssec:acid-eps0}

The \textit{layer problem} describes the fast, one-dimensional dynamics of $h>0$ for fixed $\sigma$, which acts as a parameter.
It is obtained from the fast system \cref{eq:rre-f-a} by setting $\epsilon=0$:
\begin{subequations} \label{eq:rre-lay-a}
    \begin{align}
        \frac{d \sigma}{d t } & = 0\,,  \label{eq:rre-lay-a-s} \\
        \frac{d h}{d t } & = \tilde{g}_0(\sigma, h)\,. \label{eq:rre-lay-a-h}
    \end{align}
\end{subequations}
However, $\tilde{g}_0(\sigma, h)$ is initially undefined, as $\tilde{q}_\epsilon(\sigma, h)$ becomes singular at $\epsilon = 0$, see \cref{eq:q-eps-simple}.
A continuous extension $\tilde{q}_0(\sigma, h) := \lim_{\epsilon \to 0}\tilde{q}_\epsilon(\sigma, h)$ is obtained by distinguishing the sign of $v_\epsilon(h)$, which allows us to rewrite $\tilde{q}_\epsilon(\sigma, h)$ as 
\begin{align}
    \tilde{q}_\epsilon(\sigma,h) & = \left\{ \begin{array}{ll}
          -\dfrac{\tilde{v}_\epsilon(h)}{2 \epsilon^2} \left[1 +\sqrt{1 + \epsilon^2 4 A K \tilde{r}_\epsilon(h) h^2 \sigma / \tilde{v}_\epsilon(h)^2} \right], &  \; \text{if } 0 < h < h_\epsilon^+ \,, \smallskip\\
          \epsilon^{-1}\sqrt{ A K \tilde{r}_\epsilon(h) h^2 \sigma} , &  \; \text{if } h = h_\epsilon^+ \,, \smallskip\\
           \dfrac{ 2 A K \tilde{r}_\epsilon(h) h^2 \sigma}{\tilde{v}_\epsilon(h)}\left[ 1 + \sqrt{1 + \epsilon^2 4 A K \tilde{r}_\epsilon(h) h^2 \sigma / \tilde{v}_\epsilon(h)^2} \right]^{-1}, &  \; \text{if } h > h_\epsilon^+ \,, 
    \end{array}
    \right.  \label{eq:q-case-eps-a}
\end{align}
in terms of the positive root $h_\epsilon^+$ of the equation $v_\epsilon(h)=0$:
\begin{align}
    h_\epsilon^+ = \epsilon \, \frac{-\epsilon K_h+\sqrt{ (\epsilon K_h)^2 + 4 \alpha A K K_h}}{2 \alpha A K} \,.\label{eq:h-plus-eps}
\end{align}
Thus, $h_\epsilon^+\rightarrow 0$ as $\epsilon\rightarrow 0$, and for all $h>0$ the following holds:
\begin{equation}
    \tilde{q}_0(\sigma,h) = \frac{A K \tilde{r}_0(h) h^2 \sigma}{\tilde{v}_0(h)}
                    = \frac{C \sigma}{\alpha \beta h},  \label{eq:q0-a}  
\end{equation}
upon inserting $\tilde{r}_0(h)= C/(\beta h)$ and $\tilde{v}_0(h)=\alpha A K h^2$, see \cref{eq:tilde_r-eps,eq:tilde:_v-eps}.

The so-called \textit{reduced problem} is obtained from \cref{eq:rre-fs-a} by setting $\epsilon=0$; it has the form of a differential-algebraic equation:
\begin{subequations} \label{eq:rre-red-a}
    \begin{align}
	\frac{d \sigma}{d \tau} & = \tilde{f}_0(\sigma, h) :=  -\frac{C \sigma}{\beta h} +K_s  \,,  \label{eq:rre-red-a-s} \\
	0 & = \tilde{g}_0(\sigma, h) :=  - \frac{C  \sigma}{\alpha \beta h}  + K_h(1-h)   \,.\label{eq:rre-red-a-h}
    \end{align}
\end{subequations}
The \textit{critical manifold} consists of the equilibria of the layer problem, \cref{eq:rre-lay-a-h}, and is determined by the zeros of $\tilde g_0$, given by
\begin{equation}
    \tilde{\mathcal{S}}_\textup{A} = \bigl\{ (\varphi(h), h): \, h > 0 \bigr\}
    \quad \text{with} \quad \varphi(h) := \frac{\alpha \beta K_h}{C}\, h (1-h) \,.  \label{eq:crit-man-a}
\end{equation}
The function $\varphi(h)$ possesses only one extremum at $h=1/2$, which corresponds to the acidic fold point
\begin{align}
    F_\textup{A} := (\sigma_\textup{A}, h_\textup{A}) = (\varphi(h_\textup{A}),h_\textup{A}) = \left( \frac{\alpha \beta K_h }{4 C}, \frac{1}{2} \right)\,. \label{eq:foldpoint-Fa}
\end{align}

\begin{figure}
    \includegraphics[width=0.9\textwidth]{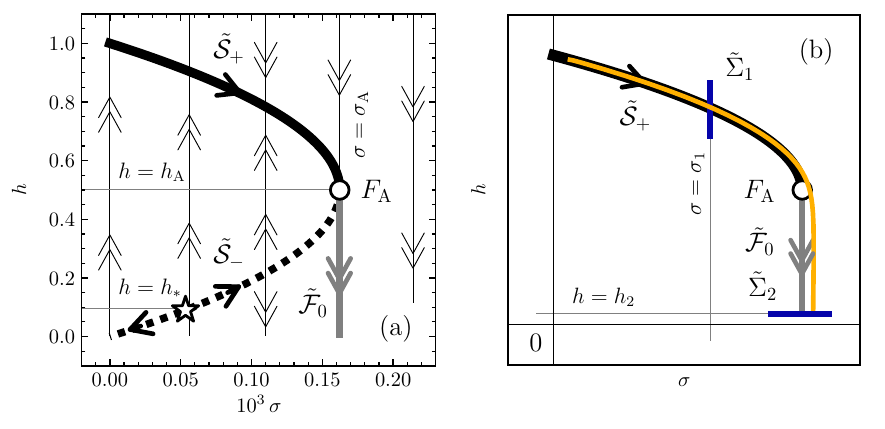}
    \caption{
    (a) The critical manifold $\tilde{\mathcal{S}}_\textup{A}=\tilde{\mathcal{S}}_+ \cup \{ F_\textup{A} \} \cup \tilde{\mathcal{S}}_-$ (thick black lines)  consists of the fold point $F_\textup{A}$ (open circle, \cref{eq:foldpoint-Fa}), the attractive branch $\tilde{\mathcal{S}}_+$ (solid line) and the repelling branch $\tilde{\mathcal{S}}_-$ (dotted line), see \cref{eq:crit-man-a}. The slow motion along $\tilde{\mathcal{S}}_+$ is governed by \cref{eq:slow-a} and progresses towards the generic fold point $F_\textup{A}$, where a loss of hyperbolicity triggers a transition to fast motion along the fast fiber. The segment of the fast fiber extending downward from $F_\textup{A}$, denoted as $\tilde{\mathcal{F}}_0$ (thick gray line), is given by \cref{eq:rre-lay-a-h} for $\sigma=\sigma_\textup{A}$. The thin horizontal gray lines indicate $h=h_\textup{A}$ and $h=h_*$; the star indicates the fixed point of system \eqref{eq:rre-red-a}.
    (b) Schematic of the $\epsilon$-perturbed extension (solid yellow line) of the critical manifold from panel~(a), following $\tilde{\mathcal{S}}_{+,\epsilon}$ and passing near the fold point $F_\textup{A}$. The Poincar{\'e} sections $\tilde{\Sigma}_1$ and $\tilde{\Sigma}_2$ are shown as dark blue lines. The thin gray lines are for $\sigma=\sigma_1$ and $h=h_2$, respectively; the remaining symbols are as in (a).
    }
    \label{fig:scaling-a}
\end{figure}

The stability of the manifold $\tilde{\mathcal{S}}_\textup{A}$
is determined by the sign of the derivative $\partial_{h} \tilde{g}_0(\sigma, h)$ on $\tilde{\mathcal{S}}_A$:
\begin{align}
    \partial_{h} \tilde{g}_0(\sigma, h)\Big|_{(\sigma, h) \in \tilde{\mathcal{S}}_\textup{A}} & = \partial_{h} \left( - \frac{C \sigma}{\alpha \beta h} + K_h(1-h)\right) \Bigg|_{(\sigma, h) \in \tilde{\mathcal{S}}_\textup{A}} \notag \\
        & = - 2 K_h \frac{h - h_\textup{A}}{h} \,. \label{eq:crit-man-a-stab}
\end{align}
In particular, $\partial_{h} \tilde{g}_0$ changes sign at $F_\textup{A}$. The critical manifold $\tilde{\mathcal{S}}_\textup{A}$ contains the sets $\tilde{\mathcal{S}}_+$ and $\tilde{\mathcal{S}}_-$, where $\tilde{\mathcal{S}}_+$ represents the attractive upper branch for $h > h_\textup{A}$, characterized by $\partial_{h} \tilde{g}_0(\sigma, h) < 0$, and $\tilde{\mathcal{S}}_-$ represents the repelling lower branch for $h < h_\textup{A}$, characterized by $\partial_{h} \tilde{g}_0(\sigma, h) > 0$, see \cref{fig:scaling-a}(a). The manifold is therefore normally hyperbolic everywhere except at the fold point $F_\textup{A}$.

The slow dynamics along the critical manifold 
$\tilde{\mathcal{S}}_\textup{A}$
is given by the reduced problem, \cref{eq:rre-red-a}. Substituting $\phi(h)$ via \cref{eq:crit-man-a} in \cref{eq:rre-red-a-s} yields:
\begin{equation}
\frac{d \sigma}{d \tau}\bigg|_{(\sigma,h) \in \tilde{\mathcal{S}}_\textup{A}}  =  \left(K_s-\frac{C\sigma}{\beta h}\right)\bigg|_{\sigma=\phi(h)} 
    % = - \alpha K_h(1-h) + K_s\,.
    = \alpha K_h(h-h_*) \,,
\end{equation}
where $h_*$ is given by \cref{eq:fixp}; in \cref{fig:scaling-a}(a), the fixed point $(K_s/\tilde{r}_0(h_*), h_*)$ of \cref{eq:rre-red-a} is shown as a star. Using $d \sigma/d \tau =d \varphi(h)/d \tau =\varphi'(h) \, d h/d \tau$ and the derivative $\varphi'(h)=(\alpha K_h \beta/C) (1-h/h_\textup{A})$, yields
\begin{align}
    \left.\frac{d h}{ d \tau}\right|_{(\sigma,h) \in \tilde{\mathcal{S}}_\textup{A}} 
    &=  \left.\frac{d \sigma}{ d \tau} \, \frac{1}{\phi'(h)} \right|_{(\sigma,h) \in \tilde{\mathcal{S}}_\textup{A}} = \left(h - 1 + \frac{K_s}{\alpha K_h} \right) \, \frac{C/\beta}{1-h/h_\textup{A}} \notag \\
    &= -\frac{C h_\textup{A}(h-h_*)}{\beta (h - h_\textup{A})} \,, \label{eq:slow-a}
\end{align}
Both, $d\sigma/d\tau$ and $dh/d\tau$ flip sign at $h=h_*$.
Restricting to $h > h_*$ and since $h_\textup{A} > h_*$, the sign of the right-hand side of \cref{eq:slow-a} is determined by the difference $h-h_\textup{A}$ and changes at $F_\textup{A}$.
For $h > h_\textup{A}$, the motion along the upper branch of the critical manifold occurs with $\diff h/ \diff \tau < 0$ from smaller to larger values of $\sigma$, whereas $\diff h/ \diff \tau > 0$ for motion on the lower branch, $h_*< h < h_\textup{A}$, also see \cref{fig:scaling-a}(a). These observations indicate that, in both cases, the motion along the attractive branch $\tilde{\mathcal{S}}_+$ and the repelling branch $\tilde{\mathcal{S}}_-$ for any $\sigma < \sigma_\textup{A}$ is always towards the non-hyperbolic point $F_\textup{A}$ (in its vicinity).

\begin{proposition} \label{prop:generic_fold_A}
    The critical manifold $\tilde{\mathcal{S}}_\textup{A}$ possesses a single fold point, given by $F_\textup{A}$ in \cref{eq:foldpoint-Fa}, and this point is a generic fold; in particular, it is nondegenerate and the flow at $F_\textup{A}$ is transversal to $\tilde{\mathcal{S}}_\textup{A}$ (see \cref{rem:gen-fold}).
\end{proposition}

\begin{proof}
We have already observed that
\begin{equation}
    \tilde{g}_0(F_\textup{A}) = 0 \quad \text{and} \quad \partial_{h}\tilde{g}_0(F_\textup{A})=0,
\end{equation}
indicating that $F_\textup{A}$ is a fold point;
it is also the only such point on $\tilde{\mathcal{S}}_\textup{A}$, see \cref{eq:rre-red-a-h}.

It remains to verify the nondegeneracy and transversality conditions on $F_\textup{A}$.
Indeed, we find from \cref{eq:rre-red-a,eq:crit-man-a-stab} that 
\begin{equation}
\label{eq:nondegen1}
        \partial_{h}^2 \tilde{g}_0(F_\textup{A}) 
        = - 4 K_h < 0, \qquad \partial_{\sigma} \tilde{g}_0(F_\textup{A})
        = - \frac{2 C}{\alpha \beta} < 0, 
\end{equation}
and 
\begin{equation}
\label{eq:nondegen2}
    \tilde{f}_0(F_\textup{A})=\alpha K_h(h_\textup{A}-h_*)>0,
\end{equation}
using $h_* > 0$ from \cref{eq:fixp}.
The conditions in \cref{eq:nondegen1} establish the nondegeneracy of the fold point. Together with the transversality condition, \cref{eq:nondegen2}, this confirms that $F_\textup{A}$ is a generic fold point.
\end{proof}

\subsubsection{Extension of slow manifolds through the fold point}
\label{ssec:acid-transit-FA}

Recall that the stable branch $\tilde{\mathcal{S}}_+$ of the critical manifold is normally hyperbolic for all $h > h_\textup{A}$.  
By the smoothness of the system, Fenichel’s theorem guarantees that for any small $\gamma,\delta>0$, the compact segment
\[ \{ (\varphi(h), h): h \in [h_\textup{A} + \gamma, 1- \delta] \}  \subset \tilde{\mathcal{S}}_+  \]
gives rise to a nearby \textit{slow manifold} $\tilde{\mathcal{S}}_{+,\epsilon}$ for sufficiently small $\epsilon>0$, lying within an $\mathcal{O}(\epsilon)$-neighborhood of the critical manifold. Although uniqueness is not guaranteed, any such slow manifold can be written, up to first-order approximation, as the graph of a function $\varphi_\epsilon(h)= \varphi(h)+ \mathcal{O}(\epsilon)$.

Standard GSPT ensures that any slow manifold can be extended beyond the fold point $F_\textup{A}$, provided it is nondegenerate. 
Such an extension is illustrated in \cref{fig:scaling-a}(b) for the dynamical system \eqref{eq:rre-f-a}.
Specifically, we consider the Poincaré sections 
\begin{equation} \label{def:tilde_Sigma}
\tilde{\Sigma}_1:= \{ (\sigma, h) : \sigma=\sigma_1,\; h\in \tilde{I} \}, \qquad 
    \tilde{\Sigma}_2:= \{ (\sigma, h) : |\sigma-\sigma_\textup{A}|\leq \delta,\; h=h_2\}
\end{equation}
for some $0< \sigma_1 < \sigma_\textup{A}$, $0< h_2 < h_\textup{A}$, and an appropriate compact interval $\tilde{I}$, such that $\tilde{\mathcal{S}}_+ \cap \tilde{\Sigma}_1$ is not empty. The extension of $\tilde{\mathcal{S}}_{+,\epsilon}$ beyond the fold point follows from the next proposition, which implies that, for any $\epsilon > 0$ sufficiently small, also the intersection $\tilde{\mathcal{S}}_{+,\epsilon} \cap \tilde{\Sigma}_1$ is nonempty and given by a point $(\sigma_1,\varphi(\sigma_1)) + \mathcal{O}(\epsilon)$.

\paragraph*{Transition map.}  
Given two transversal sections $\tilde{\Sigma}_1$ and $\tilde{\Sigma}_2$, the \textit{transition map} $\Pi^{(\epsilon)}_{1,2} : \tilde{\Sigma}_1 \rightarrow \tilde{\Sigma}_2$ assigns to each point $(\sigma, h) \in \tilde{\Sigma}_1$ the first intersection point of the corresponding trajectory of the system (depending on $\epsilon$) with $\tilde{\Sigma}_2$. That is, $\Pi^{(\epsilon)}_{1,2}(\sigma, h)$ gives the location where the trajectory starting at $(\sigma, h)$ reaches $\tilde{\Sigma}_2$. 
This map captures the evolution of trajectories between two phase space sections, which is central to understanding the global behavior of fast--slow systems.

Since $F_\textup{A}$ is a generic fold point (\cref{prop:generic_fold_A}), the following statement is a direct consequence of Theorem~2.1 by Krupa and Szmolyan \cite{krupa:SIAM-JMA2001}. 
In particular, the transversality condition, \cref{eq:nondegen2}, precludes the existence of a canard trajectory passing through $F_\textup{A}$.

\begin{proposition}    \label{PROP:transition_Fold_A}
For appropriate sections $\tilde{\Sigma}_1, \tilde{\Sigma}_2$ as defined above, the transition map $\Pi_{1,2}^{(\epsilon)}:\tilde{\Sigma}_1\rightarrow \tilde{\Sigma}_2$ is well-defined for sufficiently small $\epsilon>0$, and acts as an $\mathcal{O}(e^{-c_1/\epsilon})$-contraction for some $c_1>0$. Furthermore, for any $P\in \tilde{\Sigma}_1$, the image under the transition map satisfies
\begin{equation}
    \label{eq:hitting_fold}
   \Pi_{1,2}^{(\epsilon)}(P) = \left(\sigma_\textup{A}+\mathcal{O}\left( \epsilon^{2/3}\right),h_2\right).
\end{equation}
\end{proposition}

\subsection{The fold point \texorpdfstring{$\bm{F}_{\mathbf{B}}$}{FB} approached from the basic domain, at high \ce{pH}} \label{sec:neutral}

The candidate for a second fold point, later identified as $F_\textup{B}$, corresponds to nearly neutral $\ce{pH}$ and is approached from the basic domain at $h \simeq \epsilon C$ as $\epsilon \to 0$ (cf.~\cref{eq:foldpoint-Fb} below) and, in the limit, is pushed against the $h=0$ axis (high pH level), hindering its analysis.
To better resolve the properties of this point, we transform the dynamical system \eqref{eq:rre-eps} by rescaling the $h$ variable as
\begin{align}
   \eta :=h/\epsilon\,, \label{eq:h-resc-b}
\end{align}
which effectively zooms into the vicinity of the $h=0$ line. As a result, we arrive at the \textit{fast-slow system}:
\begin{subequations} \label{eq:rre-fs-b}
    \begin{align}
	\frac{ds}{dt} & = \hat{f}_\epsilon(s, \eta) := -\hat{r}(\eta) s + K_s \,,  \label{eq:rre-fs-b-s} \\
	\epsilon \frac{d \eta}{dt} & = \hat{g}_\epsilon(s,\eta) := -\hat{q}_\epsilon(s, \eta) + K_h (1-\epsilon \eta)
        \,,\label{eq:rre-fs-b-h}
    \end{align}
\end{subequations}
with the rescaled functions $\hat{r}(\eta) := r_\epsilon(\epsilon \eta)$ and
$\hat{q}_\epsilon(s, \eta) := q_\epsilon(s,\epsilon \eta)$, see \cref{eq:r-q-eps}. 
Introducing the fast time $t' = t / \epsilon$, the fast-slow system \eqref{eq:rre-fs-b} takes the form:
\begin{subequations} \label{eq:rre-f-b}
    \begin{align}
	\frac{ds}{d t'} & = \epsilon \hat{f}_\epsilon(s,\eta)\,,  \label{eq:rre-f-b-s} \\
	\frac{d \eta }{d t'} & = \hat{g}_\epsilon(s,\eta) \,.\label{eq:rre-f-b-h}
    \end{align}
\end{subequations}
For later use, we list the functions appearing on the r.h.s.\ of \cref{eq:rre-fs-b} for $\epsilon=0$:
\begin{equation}
    \hat{r}(\eta) = \frac{1}{\beta C \eta^{-1} + 1 + (\beta/C) \eta } \,, \label{eq:r-eps-b}
\end{equation}
which does not depend on $\epsilon$, and
\begin{equation}
    \hat{q}_0(s,\eta ) =  \frac{1}{2}\sqrt{\hat{v}_0(\eta)^2 + 4 A K \hat{r}(\eta) \eta^2 s } - \frac{\hat{v}_0(\eta)}{2} \,, \label{eq:q0-b}
\end{equation}
with
\begin{equation}
    \hat{v}_0(\eta) := v_\epsilon(\epsilon \eta)|_{\epsilon=0} =
        \alpha A K \eta^2 -K_h \,. \label{eq:v0-b}
\end{equation}

\subsubsection{Critical manifold and singular dynamics} \label{ssec:neutral-eps0}

The corresponding \textit{layer problem} is a one-dimensional dynamics in the fast variable $\eta$, parametrized by constant values of $s$. It is obtained by setting $\epsilon=0$ in \eqref{eq:rre-f-b}:
\begin{subequations} \label{eq:rre-lay-b}
    \begin{align}
    \frac{ds}{d t'} & = 0\,,  \label{eq:rre-lay-b-s} \\
    \frac{d \eta}{d t'} & =\hat{g}_0(s,\eta)\, . \label{eq:rre-lay-b-h}
    \end{align}
\end{subequations}
The \textit{reduced problem} follows by setting $\epsilon = 0$ in \eqref{eq:rre-fs-b} and reads:
\begin{subequations} \label{eq:rre-red-b}
    \begin{align}
	\frac{ds}{dt} & = \hat{f}_0(s,\eta) := -\hat{r}(\eta) s + K_s \,,  \label{eq:rre-red-b-s} \\
	0 & = \hat{g}_0(s,\eta) := -\hat{q}_0(s,\eta) + K_h \,. \label{eq:rre-red-b-h}
    \end{align}
\end{subequations}
The \textit{critical manifold} is determined by the algebraic equation \eqref{eq:rre-red-b-h}:
\begin{align}
    \hat{\mathcal{M}}_\textup{B} & := \{ (s, \eta): \, \hat{g}_0(s, \eta) = 0 \} \notag \\ 
    &= \left\{(s, \eta): (\alpha K_h - \hat{r}(\eta)s)\eta^2 = 0 \right\}\,;
\end{align}
the second line follows from the equation $2 \hat q_0(s,\eta) + \hat v_0(\eta) = 2 K_h + \hat v_0(\eta)$
after substituting \cref{eq:q0-b,eq:v0-b} and taking the square.
In particular, it is the union
$\hat{\mathcal{M}}_\textup{B} = \hat{\mathcal{S}}_\textup{B} \cup \hat{\mathcal{L}}_\textup{B}$
of a curve $\hat{\mathcal{S}}_\textup{B}$ in the domain $\eta > 0$,
\begin{equation}
    \hat{\mathcal{S}}_\textup{B} = \bigl\{ (\psi(\eta), \eta) : \eta > 0 \bigr\}
    \quad \text{with} \quad \psi(\eta) := \frac{\alpha K_h}{\hat{r}(\eta)} \,,  \label{eq:crit-man-b1}
\end{equation}
and the $s$-axis, where $\eta=0$:
\begin{equation}
    \hat{\mathcal{L}}_\textup{B} = \bigl\{ (s, 0) : s \geq 0 \bigr\}\,.  \label{eq:crit-man-b2}
\end{equation}
Inspection of the derivative
\begin{equation} \label{eq:psi-deriv}
    \psi'(\eta) = \alpha K_h (\beta/C) (1-C^2/\eta^2)
\end{equation}
shows that the branch $\hat{\mathcal{S}}_\textup{B}$ has a unique fold point,
\begin{align}
    F_\textup{B}:= (s_\textup{B}, \eta_\textup{B}) = (\psi(\eta_\textup{B}),\eta_\textup{B}) = \left( \frac{\alpha K_h}{ \hat{r}(\eta_\textup{B})}, C \right), \label{eq:foldpoint-Fb}
\end{align}
where one can further evaluate $\hat{r}(\eta_\textup{B})=1/(2\beta+1)$.

The stability of the critical manifold is inferred from the derivative of $\hat{g}_0$ along $\hat{\mathcal{S}}_\textup{B}$.
Using that $\hat g_0$ is constant on $\hat{\mathcal{S}}_\textup{B}$, it holds 
\begin{equation}
 0 = d \hat g_0(s,\eta) = \partial_s \hat g_0(s,\eta) \, ds + \partial_\eta \hat g_0(s,\eta) \, d\eta \,; \quad (s,\eta) \in \hat{\mathcal{S}}_\textup{B}.
\end{equation}
Observing further that
\begin{equation}
 \partial_s \hat g_0(s,\eta) = - \frac{A K \hat r(\eta) \eta^2}{2 \hat q_0(s,\eta) + \hat v_0(\eta)} \,,  \label{eq:dsg0-SB}
\end{equation}
one readily computes
\begin{align}
    \partial_{\eta} \hat{g}_0(s, \eta)\big|_{(s, \eta) \in \hat{\mathcal{S}}_\textup{B}}
    &= -\partial_{s} \hat{g}_0(s, \eta) \, \psi'(\eta) \big|_{(s, \eta) \in \hat{\mathcal{S}}_\textup{B}} \notag \\
    &= \frac{(\beta/C) K_h  \hat{r}(\eta)}{\eta^2 + K_h/(\alpha A K)}\left(\eta^2 - \eta_\textup{B}^2 \right), \label{eq:crit-man-b-stab}
\end{align}
after inserting \cref{eq:psi-deriv} and $\eta_\textup{B} = C$.

Hence, the stability of the manifold $\hat{\mathcal{S}}_\textup{B}$ depends on the sign of $\eta-\eta_\textup{B}$: $\hat{\mathcal{S}}_\textup{B}$ is stable for $\eta < \eta_\textup{B}$ and unstable for $\eta > \eta_\textup{B}$.
Stated differently: $\tilde{\mathcal{S}}_\textup{B}$ is normally hyperbolic everywhere, except for the point $F_\textup{B}$, where hyperbolicity is lost.
Similarly to the manifold $\tilde{\mathcal{S}}_\textup{A}$, with its fold point $F_\textup{A}$, 
the manifold $\hat{\mathcal{S}}_\textup{B}$ is divided at $F_\textup{B}$ into an attractive branch $\hat{\mathcal{S}}_+$, below $F_\textup{B}$, and a repelling branch $\hat{\mathcal{S}}_-$, above $F_\textup{B}$, which meet at the fold point, see \cref{fig:scaling-b}(a). 

\begin{figure}
    \includegraphics[width=0.9\textwidth]{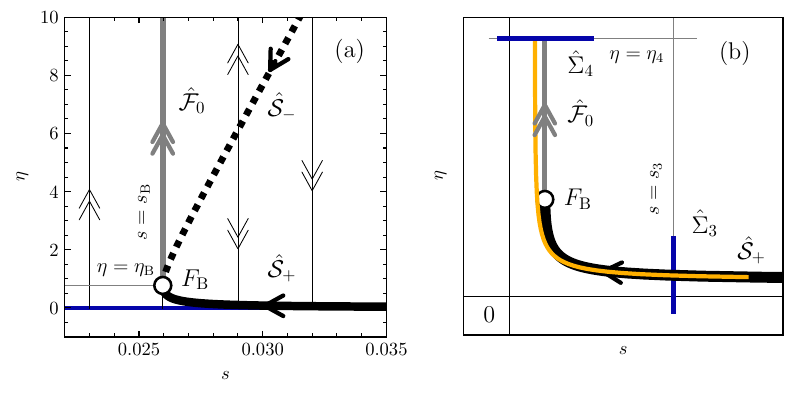}
    \caption{
    (a) The nontrivial part $\hat{\mathcal{S}}_\textup{B}=\hat{\mathcal{S}}_+ \cup \{ F_\textup{B} \} \cup \hat{\mathcal{S}}_-$ (thick black lines) of the critical manifold $\hat{\mathcal{M}}_{\textup{B}}$ consists of
    the fold point $F_\textup{B}$ (circle, \cref{eq:foldpoint-Fb}),
    the attractive branch $\hat{\mathcal{S}}_+$ (solid line), and the repelling branch $\hat{\mathcal{S}}_-$ (dotted line), see \cref{eq:crit-man-b1}.
    The trivial part $\hat{\mathcal{L}}_\textup{B}$ of $\hat{\mathcal{M}}_{\textup{B}}$ is the axis $\eta = 0$ (blue line).
    The slow motion along $\hat{\mathcal{S}}_+$ is governed by \cref{eq:slow-b} and progresses left towards the generic fold point $F_\textup{B}$, where a loss of hyperbolicity induces the switching to fast motion along the fast fiber.
    The segment of the fast fiber extending upwards from $F_\textup{B}$, denoted as $\hat{\mathcal{F}}_0$
    (thick gray line) is described by \cref{eq:rre-lay-b-h} with $s=s_\textup{B}$.
    The horizontal, thick green line marks $\eta=\eta_\textup{B}$.
    (b) Schematic of the $\epsilon$-perturbed extension (solid yellow line) of the critical manifold from panel~(a), following $\hat{\mathcal{S}}_{+,\epsilon}$ and passing near the fold point $F_\textup{B}$.
    The Poincar{\'e} sections $\hat{\Sigma}_3$ and $\hat{\Sigma}_4$ are shown as dark blue lines; the thin gray lines are for $s=s_3$ and $\eta=\eta_4$, respectively; the remaining symbols are as in (a).
    }
    \label{fig:scaling-b}
\end{figure}

The reduced flow on the critical manifold $\hat{\mathcal{S}}_\textup{B}$ follows from \cref{eq:crit-man-b1} with \cref{eq:rre-red-b-s}:
\begin{equation} \label{eq:ds-dt-SB}
    \left.\frac{ds}{dt}\right|_{(s, \eta) \in \hat{\mathcal{S}}_B} = (-\hat{r}(\eta)s + K_s)|_{s = \psi(\eta)}
        = -\alpha K_h h_* < 0 \,,% K_s - \alpha K_h\,.
\end{equation}
recalling $h_*$ from \cref{eq:fixp}.
For the fast variable, we use $ds/dt = d\psi(\eta)/dt = \psi'(\eta) \, d\eta/dt$ and \cref{eq:psi-deriv} to find:
\begin{align}
    \frac{d \eta}{ d t}\bigg|_{(s, \eta) \in \hat{\mathcal{S}}_\textup{B}} 
  &= \frac{d s}{d t} \, \frac{1}{\psi'(\eta)} \bigg|_{(s,\eta) \in \hat{\mathcal{S}}_\textup{B}} 
  = \frac{C h_*}{\beta} \frac{\eta^2}{\eta_\textup{B}^2-\eta^2} \,. \label{eq:slow-b}
\end{align}
Since $C/\beta >0$ and $h_* > 0$, the sign of $d\eta/dt$ on $\hat{\mathcal{S}}_\textup{B}$ is equal to the sign of $\eta_\textup{B}-\eta$.
Starting a trajectory on the manifold at $\eta < \eta_\textup{B}$, it evolves towards larger values of $\eta$ and is accelerated upon approaching the singular point $F_\textup{B}$, whereas the value of $\eta$ decreases upon starting at $\eta > \eta_\textup{B}$;
in both cases, the value of $s$ decreases with constant speed.
In other words, the motion along the attractive branch $\hat{\mathcal{S}}_+$ and the repelling branch $\hat{\mathcal{S}}_-$ for any $s > s_\textup{B}$ is always towards $F_\textup{B}$, see \cref{fig:scaling-b}(a). 
These observations are consistent with the following statement:

\begin{proposition} \label{prop:generic_fold_B}
The critical manifold $\tilde{\mathcal{S}}_\textup{B}$ possesses a single fold point, given by $F_\textup{B}$ in \cref{eq:foldpoint-Fb}, and this point is a generic fold.
\end{proposition}

\begin{proof}
We have already seen from \cref{eq:crit-man-b-stab} that $F_\textup{B} \in \hat{\mathcal{S}}_\textup{B}$ is a fold point, uniquely satisfying
$\hat{g}_0(F_\textup{B}) = 0$ and $\partial_{\eta}\hat{g}_0(F_\textup{B})=0$.
For the nondegeneracy of $F_\textup{B}$, one verifies the stability along the fast direction,
\begin{align}
    \partial_\eta^2 \hat{g}_0(F_\textup{B}) 
        &= \frac{ (\beta/C) K_h \hat r(\eta_\textup{B})}{\eta_\textup{B}^2 +K_h/(\alpha A K)} \, 
            \partial_\eta \left( \eta^2 -\eta_\textup{B}^2\right)\Big|_{\eta=\eta_B}  \notag \\
        &= 2\eta_\textup{B} \frac{(\beta/C) K_h \hat r(\eta_\textup{B})}{\eta_\textup{B}^2 +K_h/(\alpha A K)}>0\,,
\intertext{and regularity along the slow direction,}
     \partial_s \hat{g}_0(F_\textup{B}) % &= - \frac{A K \hat r(\eta_B) \eta_B^2}{\alpha A K \eta_B^2 +K_h} 
     &= -\frac{ \hat{r}(\eta_\textup{B})} { \alpha + K_h/(A K \eta_\textup{B}^2) } < 0 \,. \label{eq:ds-g0-FB}
\end{align}
The first equation follows from \cref{eq:crit-man-b-stab}, and we have used \cref{eq:dsg0-SB} with \cref{eq:v0-b} to obtain the second relation.
Finally, the transversality condition can be read off from \cref{eq:ds-dt-SB}:
\begin{align}
    \hat{f}_0(F_\textup{B})=- \alpha K_h h_* < 0 \,. \label{eq:f0-FB}
\end{align}
Hence, $F_\textup{B}$ is a generic fold point.
\end{proof}

\subsubsection{Extension of slow manifolds through the fold point}
\label{ssec:neutral-transit-FB}

Similarly to the situation of $\tilde{\mathcal{S}}_\textup{A}$, any compact submanifold of $\hat{\mathcal{S}}_+$ persists for $\epsilon$ sufficiently small, and can be extended beyond $F_\textup{B}$. Indeed, let
\begin{equation}\label{def:hat_Sigma}
\hat{\Sigma}_3 := \{ (s,\eta) : s = s_3, \;\eta\in \hat{I} \},  \qquad    \hat{\Sigma}_4:=\left\{ (s, \eta): \, |s-s_\textup{B}| \leq \delta, \eta = \eta_4 \right\}
\end{equation}
for some $s_3>s_\textup{B}$, $\eta_4>\eta_\textup{B}$ and an appropriate compact interval $\hat{I}$; see also \cref{fig:scaling-b}(b).
Analogously to \cref{ssec:acid-transit-FA}, the extension of any slow manifold $\hat{\mathcal{S}}_{+,\epsilon}$ beyond the generic fold point $F_\textup{B}$ (\cref{prop:generic_fold_B}) is given by the following statement, {which is granted by Theorem 2.1 of Ref.~\citenum{krupa:SIAM-JMA2001}.}
Again, a nonvanishing value of the flow in \cref{eq:f0-FB} excludes the possibility of canards at the point $F_\textup{B}$.

\begin{proposition}    \label{PROP:transition_Fold_B}
For appropriate sections $\hat{\Sigma}_3$ and $\hat{\Sigma}_4$ as defined in \cref{def:hat_Sigma}, the transition map $\Pi_{3,4}^{(\epsilon)}:\hat{\Sigma}_{3}\rightarrow \hat{\Sigma}_{4}$ is well-defined for sufficiently small $\epsilon>0$, and it acts as an $\mathcal{O}(e^{-c_2/\epsilon})$-contraction for some $c_2>0$. Furthermore, for any $P\in \hat{\Sigma}_3$, the image under the transition map satisfies
\begin{equation}
    \label{eq:hitting_fold_B}
  \Pi_{3,4}^{(\epsilon)}(P) = (s_\textup{B},\eta_4) + \mathcal{O}\left( \epsilon^{2/3}\right).
\end{equation}
\end{proposition}

\subsection{Global picture underlying the formation of the limit cycle}\label{sec:global-picture}

So far, we have clarified how two different scalings---essentially represented in \cref{fig:scheme} by the charts $K_\textup{A}$ and $K_\textup{B}$, operating in different coordinates $(\sigma, h)$ and $(s, \eta)$, respectively---resolve the dynamics around the fold points $F_\textup{A}$ and $F_\textup{B}$.
The existence of an attracting limit cycle was already established in Section~\ref{ssec:lim-cycle-exist}, and it now remains to connect these two perspectives to a global picture, revealing the formation and overall structure of the limit cycle.
Our findings from the fast–slow analysis are then illustrated and corroborated by numerical solutions for limit cycle trajectories along the $\epsilon$-perturbed slow manifolds and fast fibers.

\subsubsection{Constructing the limit cycle from transition maps}

\begin{figure}
    \includegraphics[width=0.85\textwidth]{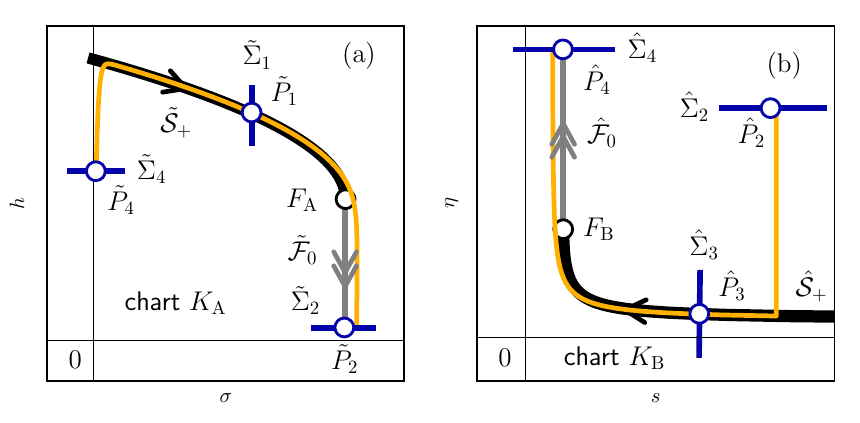}
	\caption{Composition of the limit cycle from transition maps within and between the charts $K_\textup{A}$ and $K_\textup{B}$,
	as described in \cref{sec:global-picture}.
	The two panels depict the slow attractive manifolds $\tilde{\mathcal{S}}_+$, $\hat{\mathcal{S}}_+$ (solid black line), the fast    fibers $\tilde{\mathcal{F}}_0$, $\hat{\mathcal{F}}_0$ (solid gray lines) in the singular limit, an $\epsilon$-perturbed trajectory (solid yellow lines), and suitable Poincar{\'e} sections (dark blue lines).
	(a)~Chart $K_\textup{A}$ uses coordinates $(\sigma, h)$ to show the manifold $\tilde{\mathcal{S}}_+ \cup \tilde{\mathcal{F}}_0$, which determines the passage through the fold point $F_\textup{A}$, and Poincar{\'e} sections $\tilde{\Sigma}_1$, $\tilde{\Sigma}_2$, $\tilde{\Sigma}_4$; with their intersection points $\tilde{P}_1 = \tilde{\Sigma}_1 \cap \tilde{S}_+$, $\tilde{P}_2 = \tilde{\Sigma}_2 \cap \tilde{\mathcal{F}}_0$ and $\tilde{P}_4 = \tilde{\Sigma}_4 \cap \tilde{\mathcal{F}}_0 = \{ \sigma=0 \}$.
        (b)~Chart $K_\textup{B}$ uses coordinates $(s, \eta)$ to show the manifold $\hat{\mathcal{S}}_+ \cup \hat{\mathcal{F}}_0$, which determines the passage through the fold point $F_\textup{B}$, and Poincar{\'e} sections $\hat{\Sigma}_2$, $\hat{\Sigma}_3$, $\hat{\Sigma}_4$ with intersection points $\hat{P}_3 = \hat{\Sigma}_3 \cap \hat{S}_+$ and $\hat{P}_4 = \hat{\Sigma}_4 \cap \hat{\mathcal{F}}_0$. The point $\hat P_2$ is identified with $\tilde P_2$ in chart $K_\textup{A}$.
	}
	\label{fig:global-picture}
\end{figure}

Building on the results of \cref{sec:acid,sec:neutral}, we construct a sequence of Poincar{\'e} sections and their corresponding transition maps, as illustrated in \cref{fig:global-picture}:
\begin{align*}
	\Pi_{1,2}^{(\epsilon)} &: \tilde{\Sigma}_1 \to \tilde{\Sigma}_2
    && \text{for the passage of the fold point $F_\textup{A}=(\sigma_\textup{A}, h_\textup{A})$ in chart $K_\textup{A}$,} \\
	\Phi_{\textup{A}, \textup{B}} &: \tilde{\Sigma}_2 \to \hat{\Sigma}_2
	  && \text{for the transformation from chart $K_\textup{A}$ to $K_\textup{B}$,} \\
	\Pi_{2,3}^{(\epsilon)} &: \hat{\Sigma}_2 \to \hat{\Sigma}_3
    && \text{for the transition towards the slow attracting manifold $\hat{\mathcal{S}}_+$ in chart $K_\textup{B}$,} \\
	\Pi_{3,4}^{(\epsilon)} &: \hat{\Sigma}_3 \to \hat{\Sigma}_4
    && \text{for the passage of the fold point $F_\textup{B}=(s_\textup{B}, \eta_\textup{B})$ in chart $K_\textup{B}$,} \\
	\Phi_{\textup{B}, \textup{A}} &: \hat{\Sigma}_4 \to \tilde{\Sigma}_4
    && \text{for the transformation from chart $K_\textup{B}$ to $K_\textup{A}$,} \\
	\Pi_{4,1}^{(\epsilon)} &: \tilde{\Sigma}_4 \to \tilde{\Sigma}_1
    && \text{for the transition towards the slow attracting manifold $\tilde{\mathcal{S}}_+$ in chart $K_\textup{A}$.}
\end{align*}
Here, $\hat{\Sigma}_2$ is the same as $\tilde{\Sigma}_2$ after a coordinate transformation from $(\sigma, h)$ to $(s,\eta)$, and vice versa; $\tilde{\Sigma}_4$ is $\hat{\Sigma}_4$ in $(\sigma, h)$-coordinates. The individual steps are described in detail below.
One full loop of the limit cycle is captured by the composition
$\Pi: \tilde{\Sigma}_1 \to \tilde{\Sigma}_1$ of the above transition maps:
\begin{equation}
    \Pi = \Pi_{1,2}^{(\epsilon)} \circ \Phi_{\textup{A}, \textup{B}} \circ \Pi_{2,3}^{(\epsilon)} \circ \Pi_{3,4}^{(\epsilon)} \circ \Phi_{\textup{B}, \textup{A}} \circ \Pi_{4,1}^{(\epsilon)}\,.
\end{equation}
We note that working with an arbitrarily small, but nonzero $\epsilon > 0$ is crucial for ensuring well-defined transformations between the charts $K_\textup{A}$ and $K_\textup{B}$.

\subsubsection*{Description of $\Pi_{1,2}^{(\epsilon)}$}

The transition $\Pi_{1,2}^{(\epsilon)}$ occurs within chart $K_\textup{A}$ and was characterized in \cref{PROP:transition_Fold_A}. As detailed in \cref{ssec:acid-eps0}, the upper branch $\tilde{\mathcal{S}}_+$ of the slow critical manifold $\mathcal{S}_\textup{A}$ is quickly attracting. In the singular limit, $\epsilon=0$, the intersection of $\tilde{\Sigma}_1$ with $\tilde{\mathcal{S}}_+$, see Figure~\ref{fig:scaling-a}(b), is well-defined and corresponds to a certain point $\tilde{P}_1:=(\sigma_1, h_1)=(\varphi(h_1), h_1)$ for some $h_1>0$, with $\varphi$ defined in \cref{eq:crit-man-a}. From $\tilde{P}_1$, the trajectory follows the slow manifold $\tilde{\mathcal{S}}_+$ up to the fold point $F_\textup{A}=(\sigma_\textup{A},h_\textup{A})$ (cf. \cref{eq:foldpoint-Fa}), as described in \cref{ssec:acid-eps0}, and then transitions to and along the fast fiber $\tilde{\mathcal{F}}_0$, \cref{eq:rre-lay-a-h} with $\sigma=\sigma_\textup{A}$. As a result, the trajectory follows the composite manifold $\tilde{\Gamma}_0 := \tilde{\mathcal{S}}_+ \cup \tilde{\mathcal{F}}_0$ and arrives at the point $\tilde{P}_2 := (\sigma_\textup{A}, h_2) \in \tilde{\Sigma}_2$ for the Poincar{\'e} section $\tilde{\Sigma}_2$ defined in~\cref{def:tilde_Sigma}.

For the $\epsilon$-perturbed case, the trajectory begins at a slightly shifted point, $\tilde{P}_{1,\epsilon} \in \tilde{\Sigma}_1 \cap \tilde{\mathcal{S}}_{+,\epsilon}$, follows the perturbed manifold $\tilde{\Gamma}_\epsilon = \tilde{\mathcal{S}}_{+,\epsilon} \cup \tilde{\mathcal{F}}_\epsilon$, and reaches $\tilde{P}_{2,\epsilon} := (\sigma_{\textup{A},\epsilon}, h_2) \in \tilde{\Sigma}_2$.
Both the fast fiber $\tilde{\mathcal{F}}_0$ and its perturbed continuation $\tilde{\mathcal{F}}_\epsilon$ are transverse to $\tilde{\Sigma}_2$, ensuring that the transition map $\Pi_{1,2}^{(\epsilon)}$ is well-defined.

\subsubsection*{Description of $\Phi_{\textup{A}, \textup{B}}$}

The transformation $\Phi_{\textup{A}, \textup{B}}: \tilde{\Sigma}_2 \to \hat{\Sigma}_2$ facilitates the switch from the chart $K_\textup{A}$ with the coordinates $(\sigma, h)$ to chart $K_\textup{B}$ with the coordinates $(s, \eta)$
such that $\sigma = \epsilon s$ and $\eta=h/\epsilon$, see \cref{eq:s-resc-a,eq:h-resc-b}. Hence,
the Poincar{\'e} section $\tilde{\Sigma}_2$, as defined in chart $K_\textup{A}$ by \cref{def:tilde_Sigma}, transforms into its counterpart in $K_\textup{B}$:
\begin{align} \label{hat_Sigma_2}
	\hat{\Sigma}_2:=\left\{ (s, \eta): \, |s-s_\textup{A}| \leq \delta/\epsilon, \eta = \eta_2 \right\}\,, 
\end{align}
where $s_\textup{A}:=\sigma_\textup{A}/\epsilon$ and $\eta_2:= h_2/\epsilon$. For a small enough $\epsilon$, this transformation is well-defined and unambiguous. As a result, the point $\tilde{P}_{2,\epsilon} \in \tilde{\Sigma}_2$ in chart $K_\textup{A}$ corresponds to $\hat{P}_{2,\epsilon} = (s_{\textup{A},\epsilon},\eta_2) := (\sigma_{\textup{A},\epsilon}/\epsilon, h_2/\epsilon) \in \hat{\Sigma}_2$ in chart $K_\textup{B}$.

\subsubsection*{Description of $\Pi_{2,3}^{(\epsilon)}$}

In chart $K_\textup{B}$, the manifold $\hat{\mathcal{S}}_+$ and its perturbed counterpart $\hat{\mathcal{S}}_{+,\epsilon}$ are strongly attractive, as obtained in \cref{ssec:neutral-eps0}. Starting from the point $\hat{P}_{2,\epsilon}$, the trajectory rapidly descends along the fast fiber at $s=s_{\textup{A},\epsilon}$ until it approaches the manifold $\hat{\mathcal{S}}_{+,\epsilon}$. Subsequently, it transitions to slow motion, drifting leftwards along $\hat{\mathcal{S}}_{+,\epsilon}$, see \cref{eq:ds-dt-SB}, until it intersects the Poincar{\'e} section $\hat{\Sigma}_3$ at the point $\hat{P}_{3,\epsilon} \in \hat{\Sigma}_3 \cap \hat{\mathcal{S}}_{+,\epsilon}$.

\subsubsection*{Description of $\Pi_{3,4}^{(\epsilon)}$}

This transformation describes the passage through the fold point $F_\textup{B}$, paralleling $\Pi_{1,2}^{(\epsilon)}$;
it occurs in chart $K_\textup{B}$ and was characterized in \cref{PROP:transition_Fold_B}.
The trajectory starts from the point $P_{3,\epsilon} \in \hat{\Sigma}_3\cap \hat{\mathcal{S}}_{+,\epsilon}$ and follows the manifold $\hat{\Gamma}_\epsilon = \hat{\mathcal{S}}_{+,\epsilon} \cup \hat{\mathcal{F}}_\epsilon$. Initially, the trajectory moves along the slow branch $\hat{\mathcal{S}}_{+,\epsilon}$, passing through an $\epsilon$-neighborhood of the fold point $F_\textup{B} = (s_\textup{B},\eta_\textup{B})$, \cref{eq:foldpoint-Fb}. Near the fold point, the trajectory transitions to the fast fiber $\hat{\mathcal{F}}_\epsilon$, then rapidly moving upwards to reach the point $\hat{P}_{4,\epsilon} := (s_{\textup{B},\epsilon}, \eta_4) \in \hat{\Sigma}_4$ at the Poincar{\'e} section $\hat{\Sigma}_4$ in \cref{def:hat_Sigma}.
The fast fiber $\hat{\mathcal{F}}_0$ and its perturbed continuation $\hat{\mathcal{F}}_\epsilon$ are transverse to $\hat{\Sigma}_4$, ensuring that the transition map $\Pi_{3,4}^{(\epsilon)}$ is well-defined.

\subsubsection*{Description of $\Phi_{\textup{B}, \textup{A}}$}

In analogy to $\Phi_{\textup{A}, \textup{B}}$, the transformation $\Phi_{\textup{B}, \textup{A}}: \hat{\Sigma}_4 \to \tilde{\Sigma}_4$ represents the switch from chart $K_\textup{B}$ back to chart $K_\textup{A}$. Recalling the relation between the coordinates, $\sigma = \epsilon s$ and $\eta=h/\epsilon$, the Poincar{\'e} section $\hat{\Sigma}_4$ defined in chart $K_\textup{B}$ by \cref{def:hat_Sigma} transforms in $K_\textup{A}$ into:
\begin{align}
	\tilde{\Sigma}_4:=\left\{ (\sigma, h): \, |\sigma-\sigma_\textup{B}| \leq \epsilon \delta, h = h_4 \right\}\,, \label{eq:sigma4-tilde}
\end{align}
where $\sigma_\textup{B}:= \epsilon s_\textup{B}$ and $h_4:= \epsilon \eta_4$. For $\epsilon > 0$, this transformation is well-defined and unambiguous. As a result, the point $\hat{P}_{4,\epsilon}  := (s_{\textup{B},\epsilon}, \eta_4) \in \tilde{\Sigma}_4$ in chart $K_\textup{B}$ turns into the point $\tilde{P}_{4,\epsilon} = (\sigma_{\textup{B},\epsilon}, h_4) := (\epsilon s_{\textup{B},\epsilon}, \epsilon \eta_4) \in \tilde{\Sigma}_4$ in chart $K_\textup{A}$.

\subsubsection*{Description of $\Pi_{4,1}^{(\epsilon)}$}

In chart $K_\textup{A}$, the manifold $\tilde{\mathcal{S}}_+$, and hence $\tilde{\mathcal{S}}_{+,\epsilon}$, is strongly attractive, as we have found in \cref{sec:neutral}.  This means that starting from the point $\tilde{P}_{4,\epsilon}$, the trajectory rapidly ascends along the fast fiber with $\sigma = \sigma_{\textup{B}, \epsilon}$ until it approaches the manifold $\tilde{\mathcal{S}}_{+,\epsilon}$. After that, the trajectory proceeds to slow motion, moving rightwards along $\tilde{\mathcal{S}}_{+,\epsilon}$ until it reaches the Poincar{\'e} section $\tilde{\Sigma}_1$ at the initial point $\tilde{P}_{1,\epsilon} \in \tilde{\Sigma}_1 \cap \tilde{\mathcal{S}}_{+,\epsilon}$, which closes the orbit. 

Overall, starting at the intersection point $\tilde{P}_{1,\epsilon}$ of the slow attracting manifold $\tilde{\mathcal{S}}_{+,\epsilon}$ with the Poincar{\'e} section $\tilde{\Sigma}_1$, we have demonstrated that the trajectory completes one full loop, ultimately returning to the same point $\tilde{P}_{1,\epsilon}$.

\subsubsection{Numerical results for the perturbed limit cycle} 
\label{ssec:traj-vs-gspt}

\begin{figure}
    \includegraphics[width=0.9\textwidth]{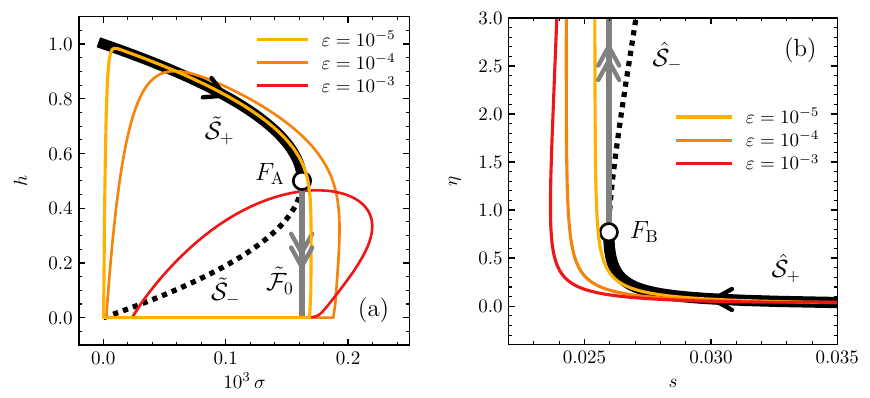}
    \caption{
    Numerical solutions of \cref{eq:rre-eps} for $\epsilon$-perturbed limit cycle trajectories for three values of $\epsilon$
    (solid colored lines) in comparison to the critical manifold (black lines) and fast fibres (gray lines) of GSPT; all other parameters are the same as in \cref{fig:limit-cycle-start,table:reduced-parameters}.
    The panels (a) and (b) correspond to the charts $K_\textup{A}$ and $K_\textup{B}$ in \cref{fig:global-picture}.
    (a)~The full limit cycles in the coordinates $(\sigma,h)$, together with the fold point $F_\textup{A}$,
    the critical manifold $\tilde{\mathcal{S}}_\textup{A}=\tilde{\mathcal{S}}_+ \cup \tilde{\mathcal{S}}_-$ (black line, \cref{eq:crit-man-a}), and the fast fiber $\tilde{\mathcal{F}}_0$ passing through $F_\textup{A}$, cf. \cref{sec:acid}.
    (b)~The region near the fold point $F_\textup{B}$ using the coordinates $(s,\eta)$,
    showing parts of limit cycle trajectories (colored lines), the critical manifold $\hat{\mathcal{S}}_\textup{B}=\hat{\mathcal{S}}_+ \cup \{ F_\textup{B} \} \cup \hat{\mathcal{S}}_-$ (black line, \cref{eq:crit-man-b1}), and the fast fiber through $F_\textup{B}$ (gray line), cf. \cref{sec:neutral}.
    }
    \label{fig:final-num}
\end{figure}

After having constructed the limit cycle for arbitrarily small $\epsilon > 0$ as a perturbation of the critical manifold by means of GSPT, we test these predictions in \cref{fig:final-num} against a few representative limit cycle trajectories, interpolating between the real system and the idealised singular limit.
The value $\epsilon=\num{e-3}$ corresponds to the experimental system \eqref{eq:rre} and decreasing $\epsilon$ increases the degree of timescale separation.
Trajectories were obtained for $\epsilon = \num{1e{-5}}, \num{1e{-4}}, \num{1e{-3}}$ via numerical integration of \cref{eq:rre-eps}
using Maple’s \texttt{dsolve} routine with the stiff solver option and an absolute accuracy of $\num{e{-10}}$.

These numerical results neatly illustrate the emergence of fast–slow dynamics as $\epsilon$ decreases: the trajectories become increasingly aligned with the attractive branches of the slow critical manifold and the fast transition layers, as predicted by GSPT. In particular, the trajectory for $\epsilon = \num{e-5}$ closely follows the $\epsilon$-perturbed limit cycle, described in \cref{sec:global-picture,fig:global-picture}.
The sharp directional changes near the fold points $F_\textup{A}$ and $F_\textup{B}$ correspond well to the fast, jump-like motions along the associated fast fibers,
while the trajectory segments in between trace the slow manifolds with high accuracy.
The trajectory for $\epsilon = \num{e-4}$ shows already pronounced deviations from the critical manifold
and is intermediate between the real system ($\epsilon = \num{e-3}$) and the idealized one ($\epsilon = \num{e-5}$), highlighting the gradual approach to the critical manifold and the role of $\epsilon$ in controlling the degree of timescale separation.
The overall behavior of the numerical solutions not only corroborates the geometric picture established by the GSPT analysis, in \cref{sec:acid,sec:neutral,sec:global-picture}, but also demonstrates how the limit cycle converges to its singular limit in a piecewise manner -- composed of alternating slow drifts and fast transitions.

\section{Implications for the chemical system} \label{sec:implications}

\subsection{Timescales and asymptotic estimates for the period} \label{sec:timescales}

One of the central practical questions related to rhythms concerns the period of oscillations and how it depends on the parameters of the system. Our GSPT analysis of the limit cycle reveals trajectory segments evolving on distinct timescales, which enable estimation of the oscillation period. In particular, we find two distinct branches of slow dynamics, evolving through the acidic domain and the other through the basic domain, see \cref{sec:acid,sec:neutral}. These intervals of slow motion alternate with episodes of fast motion, forming a full cycle of oscillation, as shown in \cref{sec:global-picture}. Under sufficient timescale separation, the contribution of the fast segments to the total period of oscillation becomes negligible, and the oscillation period is well approximated by the time spent along the slow branches. It should be noted that in the singular limit $\epsilon = 0$, the period diverges and the limit cycle ceases to exists. However, for any arbitrarily small $\epsilon > 0$, a limit cycle emerges and its period can be well approximated by the time required to traverse the corresponding branches of the critical manifold.

We now estimate the characteristic times $\mathcal{T}_\textup{acid}$ and $\mathcal{T}_\textup{basic}$ associated with the slow motion along the two branches of the limit cycle. The time $\mathcal{T}_\textup{acid}$ corresponds to motion along the perturbed slow manifold $\tilde{\mathcal{S}}_{+,\epsilon}$ in the acidic domain, approaching the $\epsilon$--vicinity of the fold point $F_\textup{A}$. Similarly, $\mathcal{T}_\textup{basic}$ describes the slow return through the basic domain along $\hat{\mathcal{S}}_{+,\epsilon}$ toward the neighborhood of the neutral fold point $F_\textup{B}$. Notably, as is inferred from \cref{fig:limit-cycle-start}(d), the approach to $F_\textup{B}$ is distinctly slower than that to $F_\textup{A}$. It is therefore of particular interest to understand the origin of this asymmetry in the dynamics.

\paragraph*{Time spent in the acidic domain, along the branch $\tilde{\mathcal{S}}_+$ (chart $K_\textup{A}$).}

Formally, the time $\mathcal{T}_\textup{acid}$ for the motion along the acidic branch of the limit cycle is defined as
\begin{align}
    \mathcal{T}_\textup{acid}(\epsilon) &:= \int_{h_\textup{start}^\textup{}}^{h_\textup{end}^\textup{}}
        \left|\epsilon \frac{d h_\epsilon}{d \tau}(\sigma,h)\right|_{(\sigma,h) \in \tilde{\mathcal{S}}_{+,\epsilon}}^{-1} \diff h\,, \label{eq:T-acid-def}
\end{align}
where $d h_\epsilon/d\tau$ is the $\epsilon$-perturbed analog of reduced dynamics in \cref{eq:slow-a} along the manifold $\tilde{\mathcal{S}}_{+,\epsilon}$. The integration over $h$ starts at $h_\textup{start}$ and ends at $h_\textup{end}$, with the latter corresponding to the perturbed point $F_{\textup{A},\epsilon} = F_\textup{A} + O(\epsilon^{2/3})$, see \cref{PROP:transition_Fold_A}. In the similar way, the starting value of integration $h_\textup{start}$ is determined by the intersection of the fast fiber through $F_{\textup{B},\epsilon}$ with the slow manifold $\tilde{\mathcal{S}}_{+,\epsilon}$.

Aiming to obtain $\mathcal{T}_\textup{acid}$ to leading order in $\epsilon$,
we approximate the perturbed manifold and fold points by their singular counterparts, which lie within an $O(\epsilon^{2/3})$ vicinity, as suggested by \cref{PROP:transition_Fold_A,PROP:transition_Fold_B}. Accordingly, the derivative in \cref{eq:T-acid-def} is approximated by $d h/d\tau$ from \cref{eq:slow-a}, evaluated along the stable branch $\tilde{\mathcal{S}}_+$ of the critical manifold $\tilde{\mathcal{S}}_\textup{A}$.
The upper limit of integration is taken to be $h_\textup{end} \approx h_\textup{A}=1/2$, corresponding to the fold point $F_\textup{A}$, see \cref{eq:foldpoint-Fa}.

The critical manifold $\tilde{\mathcal{S}}_\textup{A}$
is parametrized by the function $\phi(h)$ given by \cref{eq:crit-man-a}, and
the lower integration boundary $h_\textup{start}^\textup{}$ is determined by the condition $\varphi(h_\textup{start}^\textup{})\stackrel{!}{=}\sigma_\textup{B}$, where $\sigma_\textup{B}= \epsilon s_\textup{B}$ with $s_\textup{B}=\alpha (2\beta+1)K_h$ corresponding to the variable $s$ taken at point $F_\textup{B}$, see \cref{eq:foldpoint-Fb}.
This yields the quadratic equation $\alpha (\beta/C) K_h h_\textup{start}^\textup{} (1-h_\textup{start}^\textup{}) \stackrel{!}{=} \epsilon \alpha (2\beta +1)K_h$ with respect to $h_\textup{start}$. The larger root corresponds to the intersection with the stable part $\tilde{\mathcal{S}}_+$ of the critical manifold, and is given by $h_\textup{start}^\textup{} = {1}/{2} + \sqrt{{1}/{4} - \epsilon C (2 \beta +1)/\beta}$, which simplifies to
$h_\textup{start}^\textup{} = 1 + \mathcal{O}(\epsilon)$ for small $\epsilon$.

With these ingredients, we obtain
\begin{align}
    \mathcal{T}_\textup{acid}(\epsilon)
    &\approx \int_1^{h_\textup{A}}
        \left|\epsilon \frac{d h}{d \tau}(\sigma,h)\right|_{(\sigma,h) \in \tilde{\mathcal{S}}_+}^{-1} \diff h 
        = -\frac{2\beta}{\epsilon C} \int_1^{1/2} \frac{h-1/2}{h-h_*} \, \diff h 
    = \frac{\beta}{\epsilon C } w(h_*)       \label{eq:T-acid-dimless}
\end{align}
with 
\begin{align}
    w(h) = 1 - (1 - 2h) \ln \left(\frac{2 - 2h}{1 - 2h}\right)\,. \label{eq:w(h)}
\end{align}
Inserting $h_* \approx 0.0906$ from \cref{eq:fixp} into \cref{eq:T-acid-dimless}, we compute:
\begin{equation}
    \mathcal{T}_\textup{up}(\epsilon) \approx 0.347 \, \frac{\beta}{\epsilon C} \,. \label{eq:T-acid-dimless-num}
\end{equation}

\paragraph*{Time spent in the basic domain, along the branch $\hat{\mathcal{S}}_+$ (chart $K_\textup{B}$).}

As in the analysis of the acidic domain, we now evaluate the leading contribution to the oscillation period $\mathcal{T}_\textup{basic}$ associated with motion along the branch of the limit cycle that lies in the basic domain. We approximate the perturbed manifold by the critical manifold $\hat{\mathcal{S}}_\textup{B}$, whose stable branch $\hat{\mathcal{S}}_+$ is parametrized by the function $\psi(\eta)$, see \cref{eq:crit-man-b1}. The time $\mathcal{T}_\textup{basic}$ is obtained from \cref{eq:slow-b} by integrating over $\eta$ as
$\mathcal{T}_\textup{basic}(\epsilon) \approx \int_{\eta_\textup{start}}^{\eta_\textup{end}} [ (d \eta/ d t)|_{\hat{\mathcal S}_+} ]^{-1} \diff \eta $, where the integration limits $\eta_\textup{start}$ and $\eta_\textup{end}$
correspond to the intersections of the slow manifold $\hat{\mathcal{S}}_+$ with fast fibers passing through the fold points $F_\textup{A}$ and $F_\textup{B}$, respectively; see \cref{fig:global-picture}.

The starting point $\eta_\textup{start}$ is defined by the condition $\psi(\eta_\textup{start}^\textup{})\stackrel{!}{=} s_\textup{A}$, where $s_\textup{A}=\sigma_\textup{A}/\epsilon$ (as in \cref{hat_Sigma_2}) with $\sigma_\textup{A} = \alpha\beta K_h/(4C)$ given by the fold point $F_\textup{A}$, see \cref{eq:foldpoint-Fa}. This leads to the quadratic equation $4\epsilon\beta C \textup{start}^2 -(\beta - 4\epsilon C) \textup{start} + 4\epsilon \beta C^2 =0$, whose smaller root
corresponds to the intersection with the stable branch of the critical manifold $\hat{\mathcal S}_+$, yielding
$\eta_\textup{start} = [\beta -4\epsilon C -\sqrt{(\beta -4\epsilon C)^2-(8\epsilon\beta C)^2}]/(8\epsilon\beta)= 4\epsilon C^2 + \mathcal{O}(\epsilon^2)$. Alternatively, this result follows by approximating $\psi(\eta) \simeq \alpha \beta K_h C/\eta$ in the limit $\eta \to 0$. The end point of integration is approximated by $\eta_\textup{end} \approx \eta_\textup{B} = C$, corresponding to the fold point $F_\textup{B}$.

Combining these results, the time it takes to traverse the basic branch is evaluated to 
\begin{align}
    \mathcal{T}_\textup{basic}(\epsilon) & \approx \int_{4\epsilon C^2}^{\eta_\textup{B}} \left(\frac{d \eta}{d t}\biggl|_{\hat{\mathcal S}_+}\right)^{-1} \diff \eta =  \frac{\beta}{C h_*} \int_{4\epsilon C^2}^{\eta_\textup{B}} \frac{\eta_\textup{B}^2-\eta^2}{\eta^2} \diff \eta  \notag \\
    & = -\left.\frac{\beta}{h_*} \left(\frac{\eta}{C}+\frac{C}{\eta}\right) \right|_{4\epsilon C^2}^{C}
      =  \frac{\beta}{h_*} \frac{(1 - 4\epsilon C)^2}{4\epsilon C}
      \approx \frac{\beta }{4\epsilon C h_* } \,. \label{eq:T-basic-dimless}
\end{align}
Inserting $h_* \approx 0.0906$ (\cref{eq:fixp}), we find
\begin{align}
    \mathcal{T}_\textup{basic}(\epsilon) \approx 2.76 \, \frac{\beta }{\epsilon C}  \,. \label{eq:T-basic-dimless-num}
\end{align}

Thus, for the total oscillation period $\mathcal{T}(\epsilon)$ to the leading order in $\epsilon$, we obtain
\begin{align}
    \mathcal{T}(\epsilon) & \approx \mathcal{T}_\textup{acid}(\epsilon) + \mathcal{T}_\textup{basic}(\epsilon) 
                           = \frac{\beta}{\epsilon C} \left( \frac{1}{4 h_*} + w(h_*) \right)
                           \approx 0.311 \frac{\beta}{\epsilon C} \,.
\end{align}
In the last step, we substituted $h_* \approx 0.0906$. Note that a closer result with the prefactor $0.310$ in the above expression follows if we approximate $w(h_*)\approx 1- \ln(2) + \mathcal{O}(h_*)$ for small $h_*$.

We observe that both characteristic times, $\mathcal{T}_\textup{acid}$ and $\mathcal{T}_\textup{basic}$, are primarily governed by the same parameter combination, $\beta / (\epsilon C)$. As a result, both timescales---and thus the total oscillation period---increase as $\epsilon$ decreases. This behavior is qualitatively consistent with the fact that no limit cycle exists in the singular limit $\epsilon = 0$: the motion along the reduced manifold then requires infinite time. Geometrically, this corresponds to the expansion of the limit cycle as $\epsilon \to 0$, with its right boundary (which, in the singular limit, coincides with the acidic fold point, $F_\textup{A}$) shifting toward larger values of $s$. This trend is also evident in Figs.~\ref{fig:scheme}(a) or \ref{fig:final-num}(a), where one should recall that $\sigma = \epsilon s$. Quantitatively, the scaling $\mathcal{T}_\textup{acid}, \mathcal{T}_\textup{basic} \sim \epsilon^{-1}$ is consistent with other GSPT results, such as those obtained for the Brusselator model.\cite{engel:A2023}

The ratio between the two timescales is independent of $\epsilon$ and given by
\begin{equation}
    \frac{\mathcal{T}_\textup{acid}(\epsilon)}{\mathcal{T}_\textup{basic}(\epsilon)} \approx 4h_*w(h_*) \,, \label{eq:rel-times-dimless}
\end{equation}
where $w(h_*)$ is given by \cref{eq:w(h)}. This expression confirms that the motion along the upper branch is considerably faster, provided $h_* \ll 1$. For $h_* \approx 0.0906$ (see details following \cref{eq:fixp}), and thus the GSPT prediction yields $\mathcal{T}_\textup{acid} / \mathcal{T}_\textup{basic} \approx 0.126$. It is important to note, however, that this asymmetry between timescales arises independently of the smallness of $\epsilon$.

\paragraph*{Timescales in physical units}

In dimensionless form, the durations $\mathcal{T}_\textup{acid}$ and $\mathcal{T}_\textup{basic}$ spent along the acidic and basic branches of the limit cycle are given by expressions \cref{eq:T-acid-dimless,eq:T-basic-dimless}. To convert these to physical units, we use the rescaling relation $\epsilon C = \epsilon_1$ from \cref{eq:resc-constants} (see also \cref{sec:realistic-model} and \cref{table:parameters,table:reduced-parameters}). This yields the dimensional timescales:
\begin{align}
    T_\textup{acid} & = \frac{k_\textup{M}}{v_\textup{max}}\mathcal{T}_\textup{acid} \approx
      \frac{k_\textup{M}}{v_\textup{max}} \frac{[\ce{H+}_\textup{ext}]}{k_\textrm{E1}} w(h_*) \,, \label{eq:T-acid-dim} \\
    T_\textup{basic} & = \frac{k_\textup{M}}{v_\textup{max}}\mathcal{T}_\textup{basic} \approx
      \frac{k_\textup{M}}{v_\textup{max}} \frac{[\ce{H+}_\mathrm{ext}]}{k_\mathrm{E1} }  \cfrac{1}{4h_*}\,, \label{eq:T-basic-dim}
\end{align}
where $k_\textup{M}$ and $v_\textup{max}$ are the Michaelis--Menten constants, $k_\textup{E1}$ is the enzymatic rate constant at low pH, the function $w(h)$ is given by \cref{eq:w(h)} and the fixed-point acid concentration of system \cref{eq:rre} reads (cf. \cref{eq:fixp}):
\begin{equation}
  h_* = \frac{[\ce{H+}_*]}{[\ce{H+}_\mathrm{ext}]} = 1-2\frac{k_\textrm{S} [\ce{S}_\mathrm{ext}]}{k_\textrm{H} [\ce{H+}_\mathrm{ext}]}\,.
\end{equation}

In the regime of strong timescale separation, the total oscillation period is generally approximated by the sum $T = T_\textup{acid} + T_\textup{basic}$, which is largely governed by the compound timescale $k_\textup{M} [\ce{H+}_\mathrm{ext}]/(v_\textup{max} k_\mathrm{E1})$. For small values of
$h_*$,
as in our case with $h_* \approx 0.0906$, we find $T_\textup{basic}/T_\textup{acid} = (4 h_* w(h_*))^{-1} \approx 7.96$, see \cref{eq:rel-times-dimless,eq:w(h)}. Thus, the oscillation period is dominated by the slow excursion along the basic branch, so that $T \approx T_\textup{basic}$ as given in \cref{eq:T-basic-dim}; this observation is qualitatively consistent with the numerical data presented in \cref{fig:limit-cycle-start}(d). Generally, with the decrease in $h_*$, $T \approx T_\textup{basic}$ grows as $h_*^{-1}$ and if $h_*$ approaches zero, where oscillations would disappear, the time spent on the lower branch diverges.

\paragraph*{Comparison of analytic predictions with numerical data}

Finally, we compare the asymptotic analytic predictions obtained in the limit of strong timescale separation with numerical results from direct integration of \cref{eq:rre-eps} for $\epsilon = \num{e-5}, \num{e-4}, \num{e-3}$, using the parameter values listed in \cref{table:reduced-parameters}. The timescales spent along the slow branches in the acidic and basic domains, denoted by $\mathcal{T}_\textup{acid}(\epsilon)$ and $\mathcal{T}_\textup{basic}(\epsilon)$, are given analytically in \cref{eq:T-acid-dimless,eq:T-basic-dimless}, respectively. As previously discussed, the total oscillation period excludes the fast transitions and is approximated by $\mathcal{T}(\epsilon) \approx \mathcal{T}_\textup{acid}(\epsilon) + \mathcal{T}_\textup{basic}(\epsilon)$. From the numerical simulations, we estimate the corresponding timescales as the durations spent traveling from $F_\textup{B}$ to $F_\textup{A}$ and vice versa, denoted by $\tau_{\textup{B}\to \textup{A}}$ and $\tau_{\textup{A}\to \textup{B}}$, respectively. The full oscillation period is thus exactly $\tau = \tau_{\textup{B}\to \textup{A}} + \tau_{\textup{A}\to \textup{B}}$. The results, including the relative durations in each domain, are summarized in \cref{table:timescales}. All timescales are presented in dimensionless units; physical times (in seconds) can be obtained by dividing by the rate constant $k_\textup{max} \approx \SI{6.17e-2}{\per\second}$ (see \cref{table:parameters}).

%%%%%%%%%%%%%%%%%%%%%%%%%%%%%%%%%%%%%%%%%%%%

\begin{table}[b]
\centering
\setlength{\tabcolsep}{7pt}
\begin{tabular}{lcccccccc}
  \toprule
  $\epsilon$ & $\mathcal{T}_\textup{acid}$ & $\mathcal{\tau}_{\textup{B}\to \textup{A}}$ & $\mathcal{T}_\textup{basic}$ &
  $\mathcal{\tau}_{\textup{A}\to \textup{B}}$ & $\mathcal{T}$ & $\mathcal{\tau}$ & $\mathcal{T}_\textup{basic}$/$\mathcal{T}_\textup{acid}$ & 
  $\mathcal{\tau}_{\textup{B}\to \textup{A}}/\mathcal{\tau}_{\textup{A}\to \textup{B}}$\\
  \midrule
  \num{e-5} & 900.95  &  \num{975}  &  7174.5  &  \num{7.48e3}  &  8075.5  & \num{8.46e3} &  7.9632  &  7.67  \\  %0.1303 \\ 7482
  \num{e-4} & 90.095  &  \num{120}  &  717.45  &  825           &  807.55  &  945         &  7.9632  &  6.88  \\
  \num{e-3} & 9.0095  &  \num{23.0} &  71.745  &  67.2          &  80.755  &  90.2        &  7.9632  &  2.92  \\
  \bottomrule
\end{tabular}
\caption{Summary of the timescales obtained analytically and from numerical simulations for $\epsilon=\num{e-5}, \num{e-4}, \num{e-3}$. The values of $\mathcal{T}_\textup{acid}$ and $\mathcal{T}_\textup{basic}$ are given by \cref{eq:T-acid-dimless,eq:T-basic-dimless} with $\mathcal{T} = \mathcal{T}_\textup{acid} + \mathcal{T}_\textup{basic}$. The data correspond to numerical solution of \cref{eq:rre-eps} and is identical to those in \cref{fig:final-num}(a); all other parameters are the same as in \cref{fig:limit-cycle-start,table:reduced-parameters}. The values of $\mathcal{\tau}_{\textup{B}\to \textup{A}}$ and $\mathcal{\tau}_{\textup{A}\to \textup{B}}$ corresponds to the time spent to move from points $F_\textup{B}$ to $F_\textup{A}$ and $F_\textup{A}$ to $F_\textup{B}$, respectively, and $\tau = \tau_{\textup{B}\to \textup{A}} + \tau_{\textup{A}\to \textup{B}}$.
}.
\label{table:timescales}
\end{table}

%%%%%%%%%%%%%%%%%%%%%%%%%%%%%%%%%%%%%%%%%%%%

For smaller values of $\epsilon$ ($\epsilon = \num{e-5}, \num{e-4}$), the numerical results agree reasonably well with the theoretical predictions, with closer agreement for $\epsilon = \num{e-5}$ and slightly larger deviations at $\epsilon = \num{e-4}$. In both cases, the numerically estimated timescales $\tau_{\textup{B}\to \textup{A}}$ and $\tau_{\textup{A}\to \textup{B}}$ slightly overestimate the analytic values $\mathcal{T}_\textup{acid}$ and $\mathcal{T}_\textup{basic}$, respectively. The deviations are smaller for the basic segment. This discrepancy is expected because the analytic estimates capture only the slow drift along the critical manifold, whereas the numerical values include both slow and fast phases of motion. Consistent with theory, the ratio of time spent in the basic domain to that in the acidic domain remains nearly independent of $\epsilon$, with the theoretical value being approximately $8$. From the numerical data, we find this ratio to be $7.7$ for $\epsilon = \num{e-5}$ and $6.9$ for $\epsilon = \num{e-4}$.

For $\epsilon = \num{e-3}$, corresponding to the original system given by \cref{eq:rre}, the agreement between numerical and analytic values becomes less accurate. The time $\tau_{\textup{B}\to \textup{A}} \approx 23$ significantly overestimates the analytic value $\mathcal{T}_\textup{acid} \approx 9$, while $\tau_{\textup{A}\to \textup{B}} \approx 67$ undershoots $\mathcal{T}_\textup{basic} \approx 81$. As a result, the ratio $\tau_{\textup{A}\to \textup{B}} / \tau_{\textup{B}\to \textup{A}} \approx 3$ deviates substantially from the theoretical prediction $\mathcal{T}_\textup{basic} / \mathcal{T}_\textup{acid} \approx 8$. We attribute these larger discrepancies to the fact that the limit cycle for $\epsilon = \num{e-3}$ deviates significantly from the singular cycle, as illustrated in \cref{fig:final-num}.

\subsection{Differential transport across the membrane required for oscillations}

Our study, including its geometric decomposition, not only provides a clear explanation for the existence and structure of the limit cycle but also uncovers additional insights into the conditions supporting oscillations. Among these, we identify a refined constraint for their emergence, expressed as
$\alpha K_h > K_s$ (see \cref{rem:aK}). In terms of the original parameters, this translates to
\begin{align}
    \frac{k_\textup{H}}{k_\textup{S}} > 2 \frac{[\ce{S}_\mathrm{ext}]}{[\ce{H+}_\mathrm{ext}]}\,, \label{eq:diff-transp}    
\end{align}
where $k_\textup{H}$, $k_\textup{S}$ are transport rates across the membrane of the acid \ce{H+} and substrate \ce{S}, and $[\ce{S}_\textup{ext}]$, $[\ce{H+}_\textup{ext}]$ are their external concentrations. Condition \eqref{eq:diff-transp},
which can be interpreted as a constraint on the differential transport, critical for sustained oscillatory behavior,
sharpens requirements from prior studies \cite{bansagi:JPCB2014}, such as $ {k_\textup{H}}/{k_\textup{S}} > 1$, by explicitly linking transport rate asymmetry to external concentration ratios. Physically, this refined condition essentially imposes a condition on the relative strength of fluxes, namely: $k_\textup{H} [\ce{H+}_\mathrm{ext}] > 2 k_\textup{S} [\ce{S}_\mathrm{ext}]$.

For the stability diagram plotted in the coordinates ($k_\textup{H}/k_\textup{S}$, $[\ce{S}_\mathrm{ext}]/[\ce{H+}_\mathrm{ext}]$) -- as in \cref{fig:stability} and in Fig. 6 of Ref.~\citenum{straube:JPCB2023} -- the condition in \cref{eq:diff-transp} defines the parameter domain above the line $k_\textup{H}/k_\textup{S} = 2 [\ce{S}_\mathrm{ext}]/[\ce{H+}_\mathrm{ext}]$ (which corresponds to $K_h/K_s = \alpha^{-1}$ in \cref{fig:stability}). This inequality thus sets a lower bound for the domain of oscillations in the present reduced model, as oscillations can occur only in the parameter region above this line, where differential transport across the membrane is sufficiently strong. This analytical prediction aligns well with the numerically determined oscillatory domain reported in Ref.~\citenum{straube:JPCB2023} (cf. their Fig. 6). Notably, the numerical data from that study further suggest that \cref{eq:diff-transp} must be satisfied together with the condition $[\ce{S}_\textup{ext}] \gtrapprox [\ce{H+}_\textup{ext}]$ to support sustained oscillations.

\section{Conclusion} \label{sec:conclusion}

In this work, we have uncovered the geometric structure underlying the oscillatory dynamics of a biochemical pH oscillator,
namely, the urea–urease reaction network confined to a lipid vesicle \cite{bansagi:JPCB2014,muzika:PCCP2019,miele:JPCL2022}.
Our analysis has built on a simple yet realistic model in two variables \cite{straube:JPCB2023}
and has used geometric singular perturbation theory (GSPT) as the central mathematical tool, providing us with a natural framework for resolving the fast–slow structure of the emerging limit cycle.

Despite the inherent presence of distinct fast and slow processes, the model lacks a single dominant small parameter that would directly control the timescale separation. Instead, it is controlled by several small dimensionless parameter combinations, none of which suffices to characterize the system’s multiscale behavior. We have overcome this challenge by a suitable formal coupling of such combinations, which enables the application of classical GSPT techniques.

Using this framework, we have identified distinct critical manifolds corresponding to the slow segments of the limit cycle---one in the acidic region, the other in the basic region. To resolve their local structures, we have introduced two separate rescalings, each tailored to the respective pH domain. This approach enables a detailed analysis of the system near the associated fold points, where normal hyperbolicity is lost. In particular, we have proven that these points are generic folds rather than canards, contrary to an earlier speculation \cite{straube:JPCL2021}. A generalized GSPT, extended to handle degenerate equilibria \cite{krupa:SIAM-JMA2001}, has allowed us to rigorously track the dynamics through these non-hyperbolic regions. Finally, we have assembled a global picture of the oscillatory mechanism by matching the local analyses through suitably constructed transition maps.

One may also refrain from coupling the parameter combinations and work directly with two small parameters $\epsilon_1$ and $\epsilon_2$. This would induce the analysis of various regimes, where one of the parameters is small but fixed and the other is sent to $0$. Note that this may uncover even more subtle substructures, as demonstrated for the peroxidase-oxidase reaction via a multiscale analysis of the Olsen model \cite{kuehn:JNS2015}. One could then also analyse a double-singular limit when $\epsilon_1, \epsilon_2 \to 0$, see also Ref.~\citenum{kuehn:PD2022}. For this work, we have focused on the main structure of the fast-slow limit cycle, yielding a first approximation within a simple coherent framework.

In parallel, there is a strong motivation to
 extend the present analysis of the deterministic dynamics to stochastic regimes.
In biochemical systems, molecular fluctuations are inevitable and become increasingly pronounced as the system size decreases.
Intrinsic noise, arising from the discreteness and finite copy numbers of molecules, is known to modify the dynamics of monostable enzymatic reaction networks\cite{thomas:JCP2010,thomas:BMCSB2012}, and can even lead to noise-induced oscillations\cite{thomas:JTB2013}. In oscillatory systems, such noise effects can become especially pronounced near singularities such as fold points, where the system's sensitivity to perturbations increases \cite{Engel:CM2025}.
Indeed, an earlier numerical study by some of us \cite{straube:JPCL2021} within the framework of the chemical master equation \cite{winkelmann2020stochastic} demonstrated that periodic oscillations become increasingly irregular as noise intensifies, due to a decrease of vesicle size, eventually leading to a breakdown of rhythmicity. This suggests that the geometric structure identified through GSPT in the deterministic setting may continue to shape the dynamics as long as the noise is moderate. Supporting this idea, prior theoretical work has shown that slow manifolds persist in a probabilistic sense, and deviations due to noise can be quantified precisely \cite{berglund2006noise}.

\acknowledgements 

This research has been supported by the Deutsche Forschungsgemeinschaft
(DFG) under Germany’s Excellence Strategy -- MATH+: The Berlin Mathematics Research Center (EXC-2046/1) -- Project No. 390685689 (Subproject AA1-18), in the case of M.E. also via Subprojects AA1-8 and EF45-5.
Furthermore, M.E. thanks the DFG CRC 1114, the Einstein Foundation and the Dutch Research Council NWO (VI.Vidi.233.133) for support.

\appendix

\section{Connection with a realistic model of the urea--urease reaction} \label{sec:realistic-model}

In the following, we the relate the simplified model \cref{eq:rre} and the model derived in Ref.~\citenum{straube:JPCB2023},
which is an already reduced, yet reliable model of the full urea--urease reaction network \cite{bansagi:JPCB2014}.

Let $[\ce{S}]$ and $[\ce{H+}]$ denote the time-dependent concentrations substrate molecules (urea) and hydrogen ions within the vesicle and $[\ce{S}_\textrm{ext}]$ and $[\ce{H+}_\textrm{ext}]$ refer to their constant values in the vesicle's exterior, which acts as a reservoir.
In terms of the dimensionless variables $s=[\ce{S}]/[\ce{S}_\textrm{ext}]$ and $h=[\ce{H+}]/[\ce{H+}_\textrm{ext}]$, the reaction kinetics of the urease-catalyzed hydrolysis of urea was shown to obey the dynamical system  \cite{straube:JPCB2023}
\begin{subequations} \label{eq:rre-jpcb}
	\begin{align}
	\frac{ds}{dt} & = -k_\mathrm{cat}(s,h) \, s + k_\mathrm{S}\, (1-s)\,,  \label{eq:rre-jpcb-s} \\
	\frac{dh}{dt} & = -k \,p(s,h) \,h + k_\mathrm{H}\,(1-h)\,, \label{eq:rre-jpcb-h}
	\end{align}
\end{subequations}
for rate constants $k_\mathrm{S},k_\mathrm{H},k>0$. The effective rate $k_\mathrm{cat}(s,h) > 0$ expresses Michaelis--Menten kinetics with a pH-dependent reaction velocity and has the product form
\begin{align}
k_\mathrm{cat}(s,h)
& = k_\mathrm{cat}^\mathrm{M}(s [\ce{S}_\textrm{ext}]) \cdot f_\mathrm{H}(h[\ce{H+}_\textrm{ext}]) \label{eq:kcat-dim}
\end{align}
with
\begin{equation}
k_\mathrm{cat}^\mathrm{M}([\ce{S}]) = \frac{v_\mathrm{max} } {
	k_\mathrm{M}+[\ce{S}]}
\quad \text{and} \quad
f_\mathrm{H}([\ce{H+}]) 
    = \frac{1}{1 + [\ce{H+}]/k_\mathrm{E1} + k_\mathrm{E2} / [\ce{H+}]} \,, \label{eq:kcat-f(h)-dim}
\end{equation}
where $v_\textrm{max}$ is the maximum reaction speed, $k_\textrm{M}$ the Michaelis--Menten constant, and $k_\textrm{E1}$ and $k_\textrm{E2}$ are enzyme-specific \ce{H+} concentrations. The bell-shaped function $f_\mathrm{H}$ attains its maximum at $[\ce{H+}]_\textup{max} = (k_\textrm{E1} k_\textrm{E2})^{1/2}$ and its width scales as the ratio $\beta^{-2}:=k_\textrm{E1} / k_\textrm{E2}$.
The function $p(s,h) \geq 0$ appearing in \cref{eq:rre-jpcb-h} reads
\begin{align}
p(s,h)  & = -\frac{b(h)}{2} + \frac{1}{2}\sqrt{b(h)^2+4c(s,h)}\, \label{eq:p-jpcb}
\end{align}
with the abbreviations
\begin{subequations} \label{eq:cffs-bc-jpcb}
  \begin{align}
b(h) &:= 1 + k' [\ce{H+}_\textrm{ext}]h + ( 1 - h^{-1} ) \, k_\mathrm{H} /k \,,  \label{eq:cffs-bc-jpcb-b} \\
    c(s,h) &:= 2 \,k_\mathrm{cat}(s,h) k' [\ce{S}_\textrm{ext}] s / k \ge 0\,,  \label{eq:cffs-bc-jpcb-c}
  \end{align}
\end{subequations}
and a constant $k'>0$.

For the parameters of the model, \cref{eq:rre-jpcb,eq:kcat-dim,eq:kcat-f(h)-dim,eq:p-jpcb,eq:cffs-bc-jpcb}, we use the same values as in Ref.~\citenum{straube:JPCB2023}, which are listed in \cref{table:parameters}.
This choice of parameters introduces a timescale separation into the model, which we exploit in this work to develop the GSPT analysis. Additionally, it allows for some \emph{a priori} simplifications of the functional dependencies.

%%%%%%%%%%%%%%%%%%%%%%%%%%%%%%%%%%%%%%%%%%%%

\begin{table}[b]
\centering
%\begin{tabularx}{\textwidth}{ |X|X|X| }
\begin{tabular}{lcr@{\,}l}
  \toprule
  Parameter & Symbol & \multicolumn{2}{c}{Value}  \\
  \midrule
  external urea concentration & $[\ce{S}_\mathrm{ext}]$  &  \num{3.8e-4} & \si{M} \\
  external proton concentration & $[\ce{H+}_\mathrm{ext}]$ & \num{1.3e-4} & \si{M} \\
  \midrule
  maximum reaction speed  & $v_\mathrm{max}$ & \num{1.85e-4} & \si{M\, \s^{-1}} \\
  Michaelis-Menten constant &  $k_\mathrm{M}$ & \num{3e-3} & \si{M} \\
  enzyme constant (low pH) & $k_\mathrm{E1}$ & \num{5e-6} & \si{M} \\
  enzyme constant (high pH) & $k_\mathrm{E2}$ & \num{2e-9} & \si{M} \\
  \midrule
  ammonia protonation  rate  & $k_2$  & \num{4.3e10} &
  \si{\Molar\tothe{-1}\s\tothe{-1}}   \\
  ammonium deprotonation rate  & $k_{2r}$  & \num{2.4e1}~ & \si{\s\tothe{-1}}   \\
  \midrule
  proton transport rate  & $k_\mathrm{H}$ & \num{9e-3} & \si{\s\tothe{-1}} \\
  urea transport rate & $k_\mathrm{S}$ & \num{1.4e-3} & \si{\s\tothe{-1}} \\
  ammonia outflow rate  & $k$ &  $=k_\mathrm{S}$  \\
  ammonium ion outflow rate  & $k_+$ &  $=k_\mathrm{S}$ \\
  \midrule
  $k_2 / (k_{2r}+k_+) = $  & $k'$ & \num{1.79e9} & \si{\Molar\tothe{-1}} \\
  $v_\textrm{max} / k_\textrm{M} = $  & $k_\textrm{max}$ & \num{6.17e-2} & \si{s\tothe{-1}} \\
  \bottomrule
\end{tabular}
\caption{Parameter values of the real-world chemical system for substance concentrations, enzyme properties, reaction rates, rates of differential transport across the vesicle membrane, and derived quantities.}
\label{table:parameters}
\end{table}

\begin{table}
\centering
%\begin{tabularx}{\textwidth}{ |X|X|X| }
\begin{tabular}{c@{\:$=$\:}c@{\hspace{1em}}r@{\;}}
  \toprule
  $K_h$ & $k_\textrm{H}/k_\textrm{max}$ & \num{0.15} \\
  $K_s$ & $k_\textrm{S}/k_\textrm{max}$ & \num{0.023} \\
  $K$ & $k/k_\textrm{max}$ & \num{0.023} \\
 \midrule
  $\alpha$ & $ [\ce{H+}_\mathrm{ext}] \big/ (2[\ce{S}_\mathrm{ext}])$ &  0.17 \\
  $\beta$ & $\sqrt{k_\mathrm{E2} / k_\mathrm{E1}}$ &  \num{0.02 } \\ %\num{3.9e-4}  \\
  $\epsilon_1$ & $\sqrt{k_\mathrm{E1} k_\mathrm{E2}} \, \big/ \, [\ce{H+}_\mathrm{ext}]$ & \num{7.7e-4} \\ 
  $\epsilon_2$ & $\alpha / (k'[\ce{H+}_\mathrm{ext}])$ & \num{7.3e-7} \\ % $( 2 k'[\ce{S}_\mathrm{ext}] )^{-1}$
 \bottomrule
\end{tabular}\caption{Dimensionless parameter combinations of model \eqref{eq:rre}.
}
\label{table:reduced-parameters}
\end{table}

%%%%%%%%%%%%%%%%%%%%%%%%%%%%%%%%%%%%%%%%%%%%

To obtain the simplified model of the main text, \cref{eq:rre} together with \cref{eq:fun-r(h),eq:fun-q(s-h),eq:v(h)}, we start by noting that \cref{eq:rre-jpcb-s,eq:rre-jpcb-h} imply bounded solutions $h(t), s(t) \in [0,1]$. This follows from inspection of the signs of the r.h.s.\ of both equations and provided that the initial conditions satisfy these bounds.
Moreover, the numerical solution for the limit cycle (Ref.~\citenum{straube:JPCB2023} and \cref{fig:limit-cycle-start})
indicates that $s \lesssim 0.25$.
Hence, there is a scale separation $s \ll k_\textup{M} / [\ce{S}_\textrm{ext}] \approx 7.9$
or, equivalently, for the substrate concentration, $[\ce{S}] \ll k_\textup{M}$.
It allows us to approximately treat the Michaelis--Menten kinetics as a first-order (linear) reaction and to replace the catalytic rate by a constant, $k_\mathrm{cat}^\mathrm{M}([\ce{S}]) \approx v_\textrm{max} / k_\textrm{M} =: k_\textrm{max}$.
Additionally, the substrate concentration, over large parts of the limit cycle (cf.~\cref{fig:limit-cycle-start}), is much smaller inside of the vesicle than in its exterior, $s \ll 1$ or $[\ce{S}] \ll [\ce{S}_\textup{ext}]$, which suggests approximating the last term in \cref{eq:rre-jpcb-s} by a constant, $k_\mathrm{S}(1-s) \approx k_\mathrm{S}$, thereby simplifying the present analysis.
Regarding the function $b(h)$, we note that 
$k' [\ce{H+}_\textrm{ext}] \approx \num{2.3e5}$
is a large constant and thus 
$1 + h k' [\ce{H+}_\textrm{ext}] \approx h k' [\ce{H+}_\textrm{ext}]$,
except for tiny values of $h$,
namely, for very low proton concentration $[\ce{H+}] \lesssim 1/k'$ in the basic regime, $\mathrm{pH} \gtrsim 9.3$,
which is not reached by the limit cycle (\cref{fig:limit-cycle-start}d).

In summary, provided that the solutions obey $0 \leq s \ll \min(k_\textup{M} / [\ce{S}_\textrm{ext}], 1)$ and
$(k' [\ce{H+}_\textrm{ext}])^{-1} \ll h \leq 1$, \cref{eq:rre-jpcb} can be replaced by
\begin{subequations}
    \begin{align}
        \frac{ds}{dt} & = - k_\mathrm{max} f_\mathrm{H}(h[\ce{H+}_\textrm{ext}]) \, s + k_\mathrm{S} \,, \\
        \frac{dh}{dt} & = -k \,p(s,h) \,h + k_\mathrm{H}\,(1-h)\,.
    \end{align}
\end{subequations}
These equations coincide with the dynamical system \cref{eq:rre} written in terms of the rescaled time $t' = k_\mathrm{max} t$ (and subsequently omitting the prime) and upon introducing dimensionless parameters
$K_s := k_\mathrm{S}/k_\mathrm{max}$, $K_h := k_\mathrm{H}/k_\mathrm{max}$, and $K := k/k_\mathrm{max}$
and the functions given in \cref{eq:fun-r(h),eq:fun-q(s-h),eq:v(h)}:
\begin{align}
    r(h) &:= f_\mathrm{H}(h[\ce{H+}_\textrm{ext}]), &
    q(s,h) &:= K p(s,h) h, \quad \text{and} &
    v(h) &:= K b(h) h
\end{align}
with the parameters
\begin{subequations}
\begin{align}
    \beta &:= \sqrt{ k_\textrm{E2} / k_\textrm{E1}}, &
    \alpha &:= ~ [\ce{H+}_\mathrm{ext}] / (2 [\ce{S}_\mathrm{ext}] ) \,,  \\
    \epsilon_1 &:= \sqrt{k_\textrm{E1} k_\textrm{E2}} \big/ [\ce{H+}_\textrm{ext}], &
    \epsilon_2 &:= 1/ (2 k' [\ce{S}_\mathrm{ext}] ) \,.
\end{align} \label{def:param}
\end{subequations}
We note that, in chemistry, one often uses logarithmic variables (\ce{pS}, \ce{pH}) instead of ($s, h$), which are related to each other through
\begin{subequations}
\begin{alignat}{4}
    \ce{pS} &:= -\log_{10}([\ce{S}] / 1\,{\rm M}) && = -\log_{10}(s) - \log_{10}([\ce{S}_\mathrm{ext}] / 1\,{\rm M}) , \label{eq:logS} \\
    \ce{pH} &:= -\log_{10}([\ce{H+}] / 1\,{\rm M}) && = -\log_{10}(h) - \log_{10}([\ce{H+}_\mathrm{ext}] / 1\,{\rm M}) \,. \label{eq:logH}
\end{alignat} \label{eq:logSH}
\end{subequations}

\Cref{table:reduced-parameters} lists the values of the above parameter combinations,
for which the reduced and simplified system \eqref{eq:rre} exhibits a limit cycle, shown in \cref{fig:limit-cycle-start}.
Previous work \cite{straube:JPCB2023} and \cref{fig:stability} show that the existence of a limit cycle sensitively depends on the ratio $\alpha$ of the external concentrations, which we keep fixed here.
From \cref{table:reduced-parameters}, one identifies $\epsilon_1$ as a small parameter of the system; this ratio specifies by how much the external pH value deviates from the point, where the catalytic reaction velocity $\propto f_\mathrm{H}$ attains its maximum, which is at $\ce{pH}=7$.
Another small parameter is $\epsilon_2$, which is inversely proportional to the large parameter combination $k' [\ce{H+}_\textrm{ext}] = \alpha/\epsilon_2$ mentioned above.
The latter is the ratio between the \ce{H+} concentration in the exterior of the vesicle and the effective equilibrium constant $1/k'$ of the \ce{NH_3 + H+ <=> NH_4+} conversion reaction,
including the small outflow of \ce{NH_4+} across the membrane.
For the values given in \cref{table:reduced-parameters}, one finds that $\epsilon_1^2$ and $\epsilon_2$ are of similar order of magnitude:
$\epsilon_1^2 / \epsilon_2 \approx 0.81\,.$

\section{A primer on the analysis of fold points in GSPT}
\label{rem:gen-fold}

Here, we summarize the standard conditions that characterize \textit{fold points} and distinguish \textit{generic folds}.\cite{kuehn:book2015} Consider the slow time formulation of a fast–slow system,
$$
\dot{x} = f(x,y), \quad \epsilon \dot{y} = g(x,y) \quad (0 < \epsilon \ll 1)\,,
$$
where $x$ and $y$ denote the slow and fast variables, respectively, and $\epsilon$ is a small parameter separating their timescales.
A point $P:=(x_0,y_0)$ on the critical manifold $$\mathcal{C} := \{ (x,y) : g(x,y)=0 \}$$ is called a \textit{generic fold point} if the following conditions hold:
\begin{enumerate}
    \item \textit{Fold point conditions:}
    $$ g(P) = 0, \quad \partial_y g(P) = 0\,, $$
    which imply that $P$ lies on the critical manifold $\mathcal{C}$ and that the fast subsystem has a singularity at $P$. In particular, $\partial_y g(P) = 0$ corresponds to the loss of normal hyperbolicity of $\mathcal{C}$ at $P$, since the linearization of the fast subsystem degenerates.  

    \item \textit{Nondegeneracy conditions:}
    $$  \partial_y^2 g(P) \neq 0, \quad \partial_x g(P) \neq 0\,, $$
    where $\partial_y^2 g(P) \neq 0$ ensures that the fold is quadratic (and therefore nondegenerate) in the fast variable,
    and $\partial_x g(P) \neq 0$ guarantees that the critical manifold varies regularly with the slow variable, preventing degeneracies also in the reduced (slow) flow.

    \item \textit{Transversality condition:} 
    $$  f(P) \neq 0\,, $$
    which ensures that the slow flow crosses the fold transversally, allowing trajectories to pass between attracting and repelling branches of the critical manifold.
    Together, conditions (i)-–(iii) characterize $P$ as a \textit{generic fold point}, a structurally stable and nondegenerate singularity of the critical manifold where normal hyperbolicity is lost in a controlled manner. Away from fold points, where $\partial_y g \neq 0$, the critical manifold is normally hyperbolic. \\ 
\end{enumerate}

If any of the conditions (i)–-(iii) fails, the point 
$P$ is a \textit{nongeneric fold point}. In particular, if the transversality condition (iii) is violated so that $f(P) = 0$ 
then $P$ is called a \textit{folded singularity}. Such points represent a degeneracy where the slow flow is tangent to the fold, and this tangency is a necessary (though generally not sufficient) condition for the existence of \emph{canard} solutions\cite{benoit:CM1981,krupa:JDE2001,kuehn:book2015}---special trajectories that follow both attracting and repelling branches of the critical manifold for significant time intervals.

\bibliography{references}

\end{document}